\documentclass[10pt,reqno]{amsart}
\usepackage{amsmath,amsthm,amscd,amssymb,amsfonts, amsbsy,upref,stmaryrd}
\usepackage{latexsym}
\usepackage{mathrsfs}
\usepackage{txfonts}
\usepackage{graphicx}%
\usepackage{fancyhdr}
\usepackage{pgf}
\usepackage{color}
\usepackage{enumerate}
\usepackage{bbm}
\usepackage{exscale}

\theoremstyle{plain} 
\numberwithin{equation}{section}
\newtheorem{theorem}{Theorem}[section]
\newtheorem{corollary}[theorem]{Corollary}

\newtheorem{lemma}[theorem]{Lemma}
\newtheorem{sublemma}[theorem]{Sublemma}
\newtheorem{proposition}[theorem]{Proposition}
\theoremstyle{definition}

\newtheorem{remark}[theorem]{Remark}
\newtheorem{claim}[theorem]{Claim}
\theoremstyle{definition}
\newtheorem{definition}[equation]{Definition}

\newcommand{\dv}{\operatorname{div}}

\newcommand{\dist}{\operatorname{dist}}

\newcommand{\uh}{\mathbb{R}_{+}^{n+1}}
\newcommand{\rn}{\mathbb{R}^n}

\newcommand{\C}{\mathcal{C}}
\newcommand{\CC}{\mathbb{C}}
\newcommand{\MM}{\mathbb{M}}

\newcommand{\AC}{\mathcal{A}}
\newcommand{\SL}{\mathcal{S}}
\newcommand{\RR}{\mathbb{R}}
\newcommand{\DD}{\mathbb{D}}
\newcommand{\DDT}{\widetilde{\mathbb{D}}}

\newcommand{\AD}{\mathcal{A}^{\rm dyad}}
\newcommand{\ADQ}{\mathcal{A}_Q^{\rm dyad}}

\newcommand{\ADQG}{\mathcal{A}_{Q,\,\Gamma^{\epsilon}}^{\rm dyad}}
\newcommand{\GD}{\gamma^{\rm dyad}}
\newcommand{\GDT}{\tilde{\gamma}^{\rm dyad}}
\newcommand{\GTQM}{\widetilde{G}_{Q,M}}

\newcommand{\eps}{\varepsilon}

\newcommand{\appendicesB}{\par
\if@chapter@pp
\setcounter{chapter}{0}%
\setcounter{section}{0}%
\gdef\@chapapp{\appendixname}%
\gdef\thechapter{B\c@chapter}
\else
\setcounter{section}{0}%
\setcounter{subsection}{0}%
\gdef\thesection{B}
\fi
}
\newcommand{\appendicesA}{\par
\if@chapter@pp
\setcounter{chapter}{0}%
\setcounter{section}{0}%
\gdef\@chapapp{\appendixname}%
\gdef\thechapter{A\c@chapter}
\else
\setcounter{section}{0}%
\setcounter{subsection}{0}%
\gdef\thesection{A}
\fi
}
\def\@makechapterhead#1{      \null
     \begin{center} 
       APPENDIX \thechapter\\
       #1
      \end{center}
     \nobreak
}

\newcommand{\ree}{\mathbb{R}^{n+1}}

\begin{document}

\title{Generalized local $Tb$ Theorems for Square Functions, and applications}

\author{Ana Grau de la Herr\'an}	
\address{DEPARTMENT OF MATHEMATICS AND STATISTICS, UNIVERSITY OF HELSINKI, FINLAND}
\email{ana.grau@helsinki.fi}
\author{Steve Hofmann}\
\address{DEPARTMENT OF MATHEMATICS, UNIVERSITY OF MISSOURI, COLUMBIA}
\email{hofmanns@missouri.edu}

\begin{abstract}  A local $Tb$ theorem is an $L^2$ boundedness criterion by which  
the question of the global behavior of an operator is reduced to its local behavior,
acting on a family of test functions $b_Q$ indexed by the dyadic cubes. We present several
versions of such results, in particular, treating square function operators whose kernels do 
not satisfy the standard Littlewood-Paley pointwise estimates.   As an application
of one version of the local $Tb$ theorem, 
we show how the solvability of the Kato problem (which was implicitly based on local $Tb$ theory) 
may be deduced
from this general criterion.  We also present another version, from which we deduce
boundedness of  layer potentials associated to certain complex elliptic operators in divergence form.
\end{abstract}

\maketitle

\tableofcontents

\section{Introduction, history, preliminaries\label{c1.s1}}

The Tb theorem, like its predecessor, the T1 Theorem, is an $L^2$ boundedness criterion, originally proved by McIntosh and Meyer \cite{McM}, and by David, Journ\'e and Semmes \cite{DJS} in the context of singular integrals, but later extended by Semmes \cite{Se} to the setting of ``square functions". The latter arise in many applications in complex function theory and in PDE, and may be viewed as singular integrals taking values in a Hilbert space.  
  
 The essential idea of Tb and T1 type theorems, is that they reduce the question of $L^2$ boundedness to verifying the behavior of an operator on a single test function b (or even the constant function 1). 
  The ``local" versions that we have obtained are related to previous work of M. Christ \cite{Ch}, who proved the first local Tb theorem in the singular integral setting.  
  The term ``local" in this context, refers to the fact that, instead of one globally defined testing function b, one is allowed to test the operator locally, say on each dyadic cube, with a local testing function that is adapted to that cube.  The advantage here, in applications, is the additional flexibility that one gains:  it may be easier to verify ``good" behavior of the operator locally, when the testing functions are allowed to vary.  The point is that sometimes particular properties of the operator may be exploited to verify the appropriate testing criterion.

  Extensions of Christ's result to the non-doubling setting are due to Nazarov, Treil and Volberg \cite{NTV} and Hyt\"onen and Martikainen \cite{HyM}. For doubling measures, one can also consider more general $L^p$ type testing conditions introduced by Auscher, Hofmann, Muscalu, Tao and Thiele \cite{AHMTT}, and further studied by Hofmann \cite{H3}, Auscher and Yan \cite{AY}, Auscher and Routin \cite{AR}, Hyt\"onen and Martikainen \cite{HyM}, Hyt\"onen and Nazarov \cite{HyN}, 
  and Tan and Yan \cite{TY}.
 
   In fact, this sort of ``local Tb" criterion, in the square function setting, 
   lies at the heart of the solution of the Kato square root problem,
  and was already implicit there, see  \cite{HMc},\cite{HLMc},\cite{AHLMcT}, (and see also \cite{AT} and \cite{Se} for related results);  the connection to local $Tb$ theory is discussed in the
  survey article \cite{H1}.    
  One of the aims of the present paper is to make this connection totally explicit:
  that is, we prove a general version of the local $Tb$ theorem for square functions from which the
  solution of the Kato problem follows directly.  In particular, this requires that one replace pointwise size
  and smoothness conditions on the kernel by appropriate integral decay and orthogonality conditions.
   
Moreover, we further generalize the local $Tb$ Theorem for square functions,
to allow for relaxed size conditions on the testing functions $b_Q$
(i.e., scale invariant $L^p$ bounds, with $p>1$, rather than $p=2$).  In particular,
we extend the result of \cite{H2} in several ways, allowing for matrix-valued
$b_Q$'s, and for the removal of pointwise kernel conditions.  
We mention that another extension (of the main result of \cite{H2}) of a different sort, to the setting of 
open sets with Ahlfors-David regular boundaries,  is presented by A. Grau de la Herr\'{a}n and M. Mourgoglou \cite{GM}, and has applications to problems that connect the behavior of the harmonic measure for domains with quantitative rectifiability properties of the boundary (see \cite{HMar} and \cite{HMarUT}).

We also present a new sort of local $Tb$ theorem, one in which both the
kernels of the square function, and the testing functions themselves, are vector-valued
(previous results for vector-valued kernels have utilized matrix-valued testing functions.)
We then apply this result to give a direct proof of a recent result of Rosen \cite{R}, concerning the boundedness of layer potentials associated to divergence form complex elliptic operators,
in the half-space $\uh$,
with $t$-independent, bounded measurable coefficients.
Rosen's original proof had relied upon functional calculus results
generalizing the Kato problem (and thus also based upon 
local $Tb$ technology),  obtained in \cite{AAMc}. 
Previously,  such layer potential bounds had been known only for real or constant coefficients,
and their perturbations \cite{AAAHK}, \cite{HKMP}.

\subsection{Summary of results}  We present here a brief synopsis of the results that we prove in this paper, and provide some explanatory context.  
The precise statements of our theorems will appear in the sequel, as noted.
In these results, $\{\Theta_t\}$ will be a family of  operators mapping
$L^2(\rn,\CC^m)$, $m\geq 1$,  
uniformly into $L^2(\rn,\CC)$, for which we seek
to prove the square function bound
\begin{equation}\label{eq1.1}
\int_{\rn} \big(g_\Theta\left(f\right)(x)\big)^2\, dx:=
\iint_{\uh}|\Theta_t f(x)|^2 \frac{dtdx}{t} \, \leq \, C\,\|f\|_{L^2(\rn)}^2
\end{equation}

\smallskip
\noindent {\bf Theorem  \ref{1.pointwisetheorem}}.    We suppose that the $\CC^m$-valued
kernel of $\Theta_t$ satisfies appropriate pointwise size and smoothness conditions,
and that we are given a family $\{b_Q\}$ of matrix-valued testing functions, 
indexed by the dyadic cubes, 
which satisfies a scale invariant $L^p$ condition, for some $p>1$, and an appropriate 
accretivity condition.  Then, given local $L^p$ control on $Q$, of a localized version of
$g_\Theta(b_Q)$, we obtain \eqref{eq1.1}.
This theorem is stated and proved in Section \ref{c2}.

\smallskip
\noindent {\bf Theorem  \ref{3.genmatrixtheorem}}.  We prove a version of Theorem
 \ref{1.pointwisetheorem}, in which the pointwise kernel conditions are replaced by
appropriate integral conditions including
 ``off-diagonal decay" and quasi-orthogonality 
 in $L^2$.   In this setting, one requires
local control of a conical (as opposed to vertical) square function acting on $b_Q$.
 This theorem is stated and proved in Section \ref{c4}.
 
 \smallskip
\noindent {\it Remark}.  The point of Theorems \ref{1.pointwisetheorem} and  \ref{3.genmatrixtheorem},
is that they allow for a weaker size condition ($L^p$, with $p>1$) on the testing functions, in the
vector-valued setting (i.e., with matrix-valued testing functions).  Previous such results
had entailed either $L^2$ testing conditions, or had been restricted to the scalar-valued setting;
see, e.g., \cite{H1}, \cite{H2}, \cite{H4}.  In the case of Theorem  \ref{3.genmatrixtheorem},
the range of $p$ allowed for the testing functions is constrained by the range of ``hypercontractive"
estimates enjoyed by $\Theta_t$.
 
 \smallskip
\noindent {\bf Theorem  \ref{t4.2}}.  We prove a variant of Theorem  \ref{3.genmatrixtheorem},
in which the size condition on the testing functions is strengthened to 
require scale invariant bounds in $L^2$ (and not just $L^p$ for some $p>1$), but in which
quasi-othogonality is assumed to hold only on a subspace $H$ of $L^2(\rn,\CC^m)$.
We then obtain \eqref{eq1.1} for $f\in H$.  Here, the column vectors of the matrix valued
testing functions are assumed to belong to $H$.
 This theorem is stated and proved in Section \ref{s4}.
 
  \smallskip
\noindent {\bf Theorem  \ref{th6.1}}.  We show that the solution of the Kato square root problem
may be deduced as a consequence of Theorem \ref{t4.2}.
 This theorem is stated and proved in Section \ref{c6}.
 
   \smallskip
\noindent {\it Remark}.  As mentioned above, local $Tb$ theory was implicit in the
solution of the Kato problem.  Theorems  \ref{t4.2} and \ref{th6.1} make this connection
completely explicit.

   \smallskip
\noindent {\bf Theorem  \ref{t6.1}}.  We prove a version of Theorem  \ref{t4.2}, in 
which the testing functions are vector-valued, rather than matrix-valued.
 This theorem is stated and proved in Section \ref{s6}.

    \smallskip
\noindent {\bf Theorem  \ref{t7.1}}.   We apply Theorem \ref{t6.1} to establish
$L^2$ bounds for layer 
potentials associated to a divergence form (complex) elliptic operator $L$
in $\mathbb{R}^{n+1}_+$, with $t$-independent coefficients, assuming that 
null solutions of $L$ satisfy local H\"older continuity estimates of De Giorgi/Nash type.   
Thus, we recover a result obtained in \cite{R}, but by a direct proof which bypasses the 
functional calculus formalism elaborated in \cite{AAMc}.
 This theorem is stated and proved in Section \ref{s7}.

   \smallskip
\noindent {\it Remark}.   The novelty of Theorem \ref{t6.1} is that the testing functions are
{\it vector}-valued rather than matrix-valued.  This point of view turns out to be essential
to our approach to the application to layer potentials.
As regards Theorem  \ref{t7.1}, we point out that, at present, 
our direct treatment of layer potentials relies on estimates for
the fundamental solution, proved in \cite{HK} and \cite{AAAHK}, which assume De Giorgi/Nash type bounds.
On the other hand, the results of \cite{R} and \cite{AAMc}, in conjunction, show that
even in the absence of  De Giorgi/Nash bounds,
one may define layer potentials on $L^2$, via the functional calculus of \cite{AAMc}.
In a forthcoming paper, we plan to extend our intrinsic development of the layer potential theory to the general case,
i.e., without the  De Giorgi/Nash hypothesis.

 \subsection{Acknowledgements}

We are grateful to Tuomas Hyt\"onen, for suggesting that we try to deduce the solution of the Kato problem directly from the general theorems obtained in the dissertation of the first named author. 
We also thank Jos\'e Mar\'ia Martell, for useful discussions concerning extrapolation theory for
$A_p$ weights,  which have helped us to obtain a sharper range of exponents ``$p$" 
in Theorem  \ref{3.genmatrixtheorem}.

\subsection{Notation}\label{c1.s2}
\begin{list}{\labelitemi}{\leftmargin=1em}
\item We shall use the letters $c$, $C$ to denote positive constants, not necessarily the same at each occurrence, which depend only on dimension and the constants appearing in the hypotheses of the theorems. We shall also write $A \lesssim B$ and $A \approx B$ to mean, respectively, that $A \leq C B$ and $0 < c \leq A/B \leq C$, where the constants $c$ and $C$ are as above, unless explicitly noted.
 Moreover if we want to specify any particular dependency of the constant we will denote it by 
 subscript or by $C(\cdot)$, e.g., $C_n$ or $C(n)$ is a constant that depends on dimension $n$.
\item We  denote points in $\mathbb{R}^{n+1}$ by $(x,t)\in\mathbb{R}^n\times\mathbb{R}=\mathbb{R}^{n+1}$ (we 
use the notational convention that $x_{n+1}=t$), or sometimes, for convenience, by capital letter $X$.
\item  We set $\mathbb{R}^{n+1}_+:=\mathbb{R}^n\times(0,+\infty)$ and $\partial\mathbb{R}^{n+1}_+:=\mathbb{R}^n\times\{0\}$.
\item For a Borel set $A \subset \mathbb{R}^{n+1}$, we let $\mathbbm{1}_A$ denote the usual indicator function of $A$, i.e. $\mathbbm{1}_A(x) = 1$ if $x \in A$, and $\mathbbm{1}_A(x) = 0$ if $x \notin A$.
\item The letter $Q$ will be used to denote a cube in $\mathbb{R}^n$, and we shall write
$Q_r$ to denote that the cube has side length $\ell(Q)=r$.
\item We let $\mathbb{D}$ denote the collection of all closed dyadic cubes in $\mathbb{R}^{n}$,
and let $\mathbb{D}_k$ denote the grid of dyadic cubes of side length $2^{-k}$.
\item For a Borel set $A\subset\mathbb{R}^n$,  a Borel measure $\mu$ defined on $\rn$, and a Borel measurable function $f$,
we set $\fint_Af\,d\mu=\frac{1}{\mu(A)}\int_Af\,d\mu$.  
\item Let $q\in[1,\infty]$, we denote by $q'\in[1,\infty]$, the number such that we have $\frac{1}{q}+\frac{1}{q'}=1$, where as usual we define $1/\infty := 0$.
\item For a Banach space $X$, we let $\mathcal{B}(X)$ denote the space of bounded linear operators
on $X$.
\end{list}

\subsection{Some Standard Definitions}\label{c1.s3}

\begin{definition}
We define the Hardy Littlewood Maximal operator  $\mathcal{M}$,  acting on 
$f\in L^1_{loc}(\mathbb{R}^n)$, 
by $$\mathcal{M}(f)(x):=\displaystyle \sup_{r>0}\frac{1}{|B_r(x)|}\int_{B_r(x)}f(y)\,dy\,,$$
\noindent where $B_r(x)$ is the ball centered at x and radius r.
\end{definition}

\begin{definition}
We say that $P_t$ is a nice approximate identity, if $P_t$ is an operator of convolution type, with a smooth, compactly supported kernel $\Phi$. That means 
that for a function $f:\mathbb{R}^n\rightarrow\mathbb{C}$ 
\begin{equation*}  P_tf=\Phi_t*f \,,\, with \,\, \Phi_t=t^{-n}\Phi\left(\frac{x}{t}\right),\int\Phi(x)dx=1\,, \ \ \Phi_t\in\mathcal{C}_0^{\infty}(\mathbb{R}^n).
\end{equation*}
In particular then, $P_t1=1$, and
\begin{equation*} 
(P_tf)(x)\leq C\mathcal{M}f(x)\,.
\end{equation*}
\label{Ptdefinition}\end{definition}

\begin{definition}
For $0<s\leq 1$, the homogeneous Sobolev space $\dot{L}^2_s$ is the completion of $C_0^\infty$ with respect to the norm
$\|f\|_{\dot{L^2_s}}:=\|(-\Delta)^{s/2}) f\|_{L^2},$ where $\Delta$ is the usual Laplacian.

\end{definition}

\begin{definition}\cite{St2}
 If $0<\alpha<n$, then the Riesz potential $I_{\alpha}f$ of a locally integrable function f on $\mathbb{R}^n$ is the function defined by
\begin{equation*}
(-\Delta)^{-\alpha/2}f(x)
=I_{\alpha}f(x)=\frac{1}{C_{\alpha,n}}\int_{\mathbb{R}^n}\frac{f(y)}{|x-y|^{n-\alpha}}dy
\end{equation*}
where the constant is given by $C_{\alpha,n}=\pi^{n/2}2^{\alpha}\frac{\Gamma(\alpha/2)}{\Gamma((n-\alpha)/2)}$.
This fractional integral is well-defined provided f decays sufficiently rapidly at infinity, specifically if $f\in L^p(\mathbb{R}^n)$ with $1\leq p<\frac{n}{\alpha}$.  
\label{Rieszpotential}\end{definition}








In the sequel, we shall use the following result.
\begin{lemma}\cite[Lemma 3.11]{AAAHK}\label{lemma3.11AAAHK}
Suppose that $\theta_t$ is an operator satisfying
\begin{equation}\label{eq1.gaff}
\Vert\theta(f\mathbbm{1}_{2^{k+1}Q\setminus 2^kQ})\Vert_{L^2(Q)}^2\leq C2^{-(n+2)k}\,
\Vert f\Vert_{L^2(2^{k+1}Q\setminus 2^kQ)}\,,
\end{equation}
whenever $t\approx \ell(Q)$, and that $\Vert\theta_t\Vert_{2->2}\leq C$. Let $b\in L^{\infty}(\mathbb{R}^n)$, and let $\mathcal{A}_t$ denote a self-adjoint averaging whose kernel $\varphi_t(x)$ satisfies $|\varphi_t(x)|\leq Ct^{-n}\mathbbm{1}_{\{|x|<Ct\}}$, $\varphi_t\geq0$,$\int\varphi_t(x)dx=1$. Then
\begin{equation*}
\displaystyle\sup_{t>0}\Vert(\theta_tb)\mathcal{A}_tf\Vert_{L^2(\mathbb{R}^n)}\leq C\Vert b\Vert_{L^{\infty}(\mathbb{R}^n)}\Vert f\Vert_{L^2(\mathbb{R}^n)}.
\end{equation*}
\end{lemma}

\section{Local Tb Theorem for Square functions with vector-valued kernels}\label{c2}

 In this section we extend the main theorem of \cite{H2} to the setting where
the kernels take values in $\CC^m,\, m\geq 1$ (the case $m=1$ is the result of
\cite{H2}).   This entails that the testing functions are now matrix valued.
To handle
this more general situation,
we follow the sectorial decomposition technique in \cite{HLMc} and \cite{AHLMcT}, but 
in treating matrix valued testing functions belonging only to $L^p$ with $p<2$,
there are certain technical difficulties
which are not present either in the case $p=2$, or in the scalar case for any $p$.
We let $\mathbb{M}^{m}$ denote the $m\times m$ matrices with complex entries.
Here, $m$ and $n$ are not required to be equal.

\begin{definition}

Suppose that $\Psi_t=(\Psi_t^1,...,\Psi_t^m):\mathbb{R}^n\times\mathbb{R}^n\rightarrow\mathbb{C}^{m}$  satisfies the following properties for some exponent $\alpha>0$

\begin{equation}
|\Psi_t(x,y)|\leq C\frac{t^{\alpha}}{(t+|x-y|)^{n+\alpha}}, 
\label{1.pointwisebound}\end{equation}

\begin{equation}
|\Psi_t(x,y+h)-\Psi_t(x,y)|+|\Psi_t(x+h,y)-\Psi_t(x,y)|\leq C\frac{|h|^{\alpha}}{(t+|x-y|)^{n+\alpha}}\,,
\label{1.pointwisecontinuity}\end{equation}
whenever 
$|h|\leq t/2$.
Then  for vector valued $f:\mathbb{R}^n\rightarrow\mathbb{C}^m$, we define
the operator
\begin{equation}
\Theta_t f(x)=\int_{\mathbb{R}^n}\Psi_t(x,y)\cdot f(y)\,dy 
:=\sum_{j=1}^m \int_{\mathbb{R}^n}\Psi_t^j(x,y)\, f_j(y)\,dy\,.
\label{1.pointwisesquarefunction}\end{equation}
 We also define $\Theta_t$ acting on matrix valued
$b=(b_{ij})_{1\leq i,j\leq m}:\RR^n\to\MM^m$ in the obvious way, 
by viewing the kernel $\Psi_t(x,y)$ as a $1\times m$ matrix which multiplies the $m\times m$ matrix
$b$, i.e.,  
\begin{equation}
\Theta_t b(x)=\left(\sum_{i=1}^m \int_{\mathbb{R}^n}\Psi_t^i(x,y)\, b_{ij}(y)\,dy\right)_{\! 1\leq j\leq m}\,.
\label{1.pointwisesquarefunction2}\end{equation} 
\end{definition}

\begin{theorem}\label{1.pointwisetheorem}

We define  $\Theta_t$ as above and suppose that there exists  constants 
$\delta>0$ and $C_0<\infty$, an exponent $p>1$,  and a system $\{b_Q\} \subset L^p(\rn,\MM^m),$ indexed by dyadic cubes $Q\subset\mathbb{R}^n$, such that for each dyadic cube Q, we have
\begin{equation}
\int_{\mathbb{R}^n}|b_Q(x)|^pdx\leq C_0|Q|,
\label{1.pbQcondition}\end{equation}
\begin{equation}
\int_Q\left(\int_0^{\ell(Q)}|\Theta_tb_Q(x)|^2\frac{dt}{t}\right)^{\frac{p}{2}}dx\leq C_0|Q|,
\label{1.thetabQcondition}\end{equation}
\begin{equation}
\delta|\xi|^2\leq \mathcal{R}e\left(\xi\cdot\fint_Qb_Q(x)dx\,\bar{\xi}\right), \ \forall\xi\in\mathbb{C}^m\,.
\label{1.ellipticitybQcond}\end{equation}
Then
\begin{equation}
\iint_{\mathbb{R}_+^{n+1}}|\Theta_t f(x)|^2\frac{dxdt}{t}\leq C||f||_2^2.
\label{1.conclusion}\end{equation}

\end{theorem}

The outline of the proof goes as follows.   By the T1 Theorem of \cite{CJ}, Theorem \ref{1.T1theorem}, we reduce matters to showing that our operator satisfies the Carleson measure estimate \ref{1.Carlesonmeasurestimate}. Then the proof has three steps: 1) the conditions of Theorem \ref{1.pointwisetheorem} imply the conditions of Lemma \ref{1.lemma}; 2)  the conditions of Lemma \ref{1.lemma} imply the conditions of Sublemma \ref{1.sublemma}. Finally Sublemma \ref{1.sublemma} establishes the Carleson measure estimate \ref{1.Carlesonmeasurestimate}, which by the T1 theorem leads to our conclusion.

Let us first state these results, and then we will start with the proofs.

\begin{theorem} (T1 Theorem of \cite{CJ}).\label{1.T1theorem}
Let $\Theta_tf(x)\equiv\int_{\mathbb{R}^n}\Psi_t(x,y)\cdot f(y)dy$, where the
kernel $\Psi_t(x,y)$ satisfies conditions (\ref{1.pointwisebound}) and (\ref{1.pointwisecontinuity}) as above.  Suppose that we have the Carleson measure estimate
\begin{equation}
\displaystyle\sup_Q\frac{1}{|Q|}\int_0^{\ell(Q)}\int_Q|\Theta_t1(x)|^2\frac{dxdt}{t}\leq C\,,
\label{1.Carlesonmeasurestimate}\end{equation}
where ``1" in this context denotes the $m\times m$ identity matrix.
Then we have the square function estimate
\begin{equation}
\iint_{\mathbb{R}_+^{n+1}}|\Theta_t f(x)|^2\frac{dxdt}{t}\leq C||f||_2^2\,.
\end{equation}
\end{theorem}

In the sequel, we shall work with cones in
$\mathbb{C}^m$, which we identify with $\mathbb{R}^{2m}$,
having vertex at the origin.  Given a unit vector
$\nu \in \mathbb{C}^m$,  and $\alpha >0$,  we let 
$\Gamma^\alpha (\nu)$ denote the cone of aperture $\alpha$ and central axis $\nu$, i.e.,
\begin{equation}
\label{eq2.conedef}
\Gamma^{\alpha}(\nu):=\left\{z\in\mathbb{C}^m:\left|\frac{z}{|z|}-\nu\right|<\alpha\right\}\,.
\end{equation} 
Sometimes, when working with a fixed cone,
we shall simply write $\Gamma^\alpha$, leaving the direction vector $\nu$ implicit.
We  let ${\bf 1}_{\Gamma^\alpha}$ denote the indicator function of $\Gamma^\alpha$,
i.e., ${\bf 1}_{\Gamma^\alpha}(z) =1$ if $z\in\Gamma^\alpha$, and ${\bf 1}_{\Gamma^\alpha}(z) =0$
otherwise.

Given a small $\epsilon>0$, 
we cover $\mathbb{C}^m$ by cones of aperture $\epsilon$, enumerating these cones as 
$\Gamma_1^{\epsilon},...,\Gamma_K^{\epsilon}$, where $K=K(\epsilon,m)$.  In the sequel, 
we shall also consider the ``doubled" cones $\Gamma_k^{2\epsilon}$, $1\leq k\leq K$,  
each with the same direction
vector as the original one, but with the aperture $2\epsilon$.

\begin{lemma}\label{1.lemma}
Suppose that there exists $\eta\in(0,1)$, $\epsilon>0$ small and $C_1<\infty$, such that for 
each cone $\Gamma^\epsilon$, and for
every dyadic cube $Q\in\mathbb{R}^n$, there is a family $\{Q_j\}$ of non-overlapping dyadic sub-cubes of Q, satisfying
\begin{equation}\label{1.cond.1}
\sum_j|Q_j|\leq(1-\eta)|Q|
\end{equation}
\noindent and
\begin{equation}\label{1.cond.2}
\int_Q\left(\int_{\tau_Q(x)}^{\ell(Q)}|\Theta_t1(x)|^2\mathbbm{1}_{\Gamma^{2\epsilon}}(\Theta_t1(x))\frac{dt}{t}\right)^{\frac{p}{2}}dx\leq C_1|Q|\, ,
\end{equation} 
where $\tau_Q(x)=\sum_j\ell(Q_j)\mathbbm{1}_{Q_j}(x) $.
Then we have the Carleson Measure estimate (\ref{1.Carlesonmeasurestimate}).

\end{lemma}

\begin{sublemma}\label{1.sublemma}

Suppose that $\exists \,N<+\infty$ and $\beta\in(0,1)$ such that for every dyadic cube $Q$, and for 
each cone $\Gamma^{\epsilon}$,  we have
\begin{equation}\label{eq2.14}
|\{x\in Q:g_Q(x)>N\}|\leq (1-\beta)|Q|,
\end{equation}
where
\begin{equation}\label{eq2.15}
g_Q(x):=\left(\int_0^{\ell(Q)}|\Theta_t1(x)|^2\mathbbm{1}_{\Gamma^{2\epsilon}}(\Theta_t1(x))\frac{dt}{t}\right)^{\frac{1}{2}}.
\end{equation}
Then we have the Carleson Measure estimate (\ref{1.Carlesonmeasurestimate}).
\end{sublemma}

\begin{remark}
Every $g_Q$ also depends on the cone of definition but since we are choosing a generic cone we avoid complicating the notation by adding more indices.
\end{remark}

\subsection{Step 1:  Hypotheses of Theorem \ref{1.pointwisetheorem} imply hypotheses of Lemma \ref{1.lemma}}\label{c2.s1}

We may assume without loss of generality that $1<p<2$, as the case $p>2$ may be reduced to the known case $p=2$ by H\"older's inequality.  The case $p=2$ is proved in \cite{H1}.

 For any cube $Q$, let 
$$R_Q:= Q\times \big(0,\ell(Q)\big)$$ denote the standard Carleson box above $Q$, and
let $A_t$ denote the usual dyadic averaging operator, i.e.,
$$A_tf(x):=\fint_{Q(x,t)}f(y)\,dy\,,$$
where $Q(x,t)$ is the minimal dyadic cube 
containing $x$ with side length at least $t$. 

The deep fact underlying Step 1 is the following.  
\begin{lemma}  \label{l2.6} Fix a dyadic cube $Q$,  a cone $\Gamma^\epsilon$,
and its double $\Gamma^{2\epsilon}$.  
Suppose that $b_Q$ satisfies \eqref{1.pbQcondition} and \eqref{1.ellipticitybQcond}.
Then for $\epsilon>0$, sufficiently small, depending only on the constants $\delta$ and $C_0$,
there is a family $\{Q_j\}$ of non-overlapping dyadic sub-cubes of $Q$, satisfying  \eqref{1.cond.1},
such that 
\begin{equation}\label{eq2.17aaa}
|\Theta_t1(x)|^2 \mathbbm{1}_{\Gamma^{2\epsilon}}\big(\Theta_t1(x)\big)
\leq 4|\Theta_t1(x) A_{t}b_Q(x)|^2\,,\qquad \forall (x,t) \in R_Q\setminus
\Big(\bigcup_j R_{Q_j}\Big)\,.
\end{equation} 
\end{lemma}

\begin{proof}[Proof of Lemma \ref{l2.6}]  The proof follows that of the analogous step
in the solution of the Kato square root problem (cf. \cite{HMc}, \cite{HLMc}, \cite{AHLMcT}).
We first construct the appropriate family of non-overlapping dyadic subcubes, using a stopping time argument.

Without loss of generality (by renormalizing), we may assume $\delta\equiv 1$ in 
\eqref{1.ellipticitybQcond} (of course, this changes $C_0$, depending on $\delta$). 
Fix a cube Q, and a cone $\Gamma^{2\epsilon}$. 
We subdivide Q dyadically and select a family $\{Q_j\}$ of non-overlapping dyadic subcubes
of $Q$, which are 
maximal with respect to the property that at least one of the following conditions holds:
\begin{equation}\label{eq2.16a}
\frac{1}{|Q_j|}\int_{Q_j}|b_Q(x)|dx\geq\frac{1}{8\epsilon} \,\qquad  {\rm (type \ I) }
\end{equation}
\begin{equation}\label{eq2.17a}
\mathcal{R}e\left(\nu\cdot\fint_{Q_j}b_Q(x)dx\,\bar{\nu}\right)\leq\frac{3}{4} \qquad {\rm(type \ II)}
\end{equation}
where $\nu$ is the unit vector in the direction of the central axis of the cone $\Gamma^{2\epsilon}$,
i.e., $$\Gamma^{2\epsilon}:=\left\{z\in\mathbb{C}^m:
\left|\frac{z}{|z|}-\nu\right|<2\epsilon\right\}\,.$$

Having constructed the family, let us first verify that it satisfies the required 
condition (\ref{1.cond.1}).
Define $E:=Q\setminus \{\bigcup_j Q_j\}$. Then from condition (\ref{1.ellipticitybQcond}), since $\delta\equiv 1$, and taking $\xi=\nu$,  
we have
\begin{align*}
|Q|\,\leq\, &\, \mathcal{R}e\left(\nu\cdot\int_Q b_Q(x)dx\,\bar{\nu}\right)\\
=\,&\, \mathcal{R}e\left(\nu\cdot\int_E b_Q(x)dx\,\bar{\nu}\right)
\,+\,\mathcal{R}e\sum_j\left(\nu\cdot\int_{Q_j} b_Q(x)dx\,\bar{\nu}\right)
\,:=\,I+II.
\end{align*}
Since $\nu$ is a unit vector, using condition (\ref{1.pbQcondition}) and H\"older's inequality we get
\begin{equation*}
I\, \leq \,\int_E\left|b_Q(x)\right|dx
\,\leq\,\left|E\right|^{\frac{1}{p'}}\left(\int_Q|b_Q(x)|^pdx\right)^{\frac{1}{p}}
\,\leq\, C |E|^{\frac{1}{p'}}|Q|^{\frac{1}{p}}.
\end{equation*}
For the second part we separate the family of subcubes $\{Q_j\}$ into
 two cases:  the ones that satisfy the type I condition and the ones that satisfy the type II condition (the same subcube can satisfy both conditions at the same time; in this case we arbitrarily assign them to be of type I).  We have
\begin{equation*}
II\,=\, \mathcal{R}e\displaystyle\sum_{j,\,type \ I}\left(\nu\cdot\int_{Q_j} b_Q(x)dx\,\bar{\nu}\right)+
\mathcal{R}e\displaystyle\sum_{j,\,type \ II}
\left(\nu\cdot\int_{Q_j} b_Q(x)dx\,\bar{\nu}\right)=:II_1+II_2.
\end{equation*}
Set
$$B_1:= \cup_{k,type \,I} \,Q_k\,,\qquad B_2:= \cup_{k,type\, II} \,Q_k\,.$$
For the type I subcubes we apply H\"older's inequality, and condition \eqref{1.pbQcondition}, to obtain
\begin{align}\nonumber
\left|II_1\right| &\leq\displaystyle\sum_{j,\,type\, I}\int_{Q_j}|b_Q(x)|dx
\,=\,\displaystyle\int_{B_1}|b_Q(x)|dx\\\label{eq2.16}
&\leq\left(\int_Q|b_Q(x)|^pdx\right)^{\frac{1}{p}}\left|B_1\right|^{\frac{1}{p'}}
\,\lesssim\, |Q|^{1/p}\,|B_1|^{1/p'}\,.
\end{align}
For the measure of $B_1$, by the definition of type I, and  \eqref{1.pbQcondition},
we have that
\begin{align*}
|B_1|=\sum_{j,\,type\,I}\left|Q_j\right| &\leq 
 \,\sum_{j,\,type\, I}8\epsilon\int_{Q_j}|b_Q(x)|dx\\
&=\,8\epsilon \int_{B_1}|b_Q(x)|dx
\leq \,8\epsilon\,|B_1|^{\frac{1}{p'}}\left(\int_Q|b_Q(x)|^pdx\right)^\frac{1}{p}\\
&\lesssim\,8\epsilon\, |B_1|^{1/p'} |Q|^{1/p}\,,
\end{align*}
whence it follows that
\begin{equation*}
 \left|B_1\right|\,\lesssim\, \epsilon^p \,|Q|
\end{equation*}
Combining this bound with \eqref{eq2.16}, and choosing $\epsilon$ small enough,
we have
\begin{equation}
\left|II_1\right|\,\leq\,C\epsilon^{\frac{p}{p'}}|Q|\,\leq \frac18 |Q|\,.\label{1.fixepsilon}
\end{equation}
By the definition of the type II cubes,  we have
$$\left|II_2\right|:=\left|\mathcal{R}e
\displaystyle\sum_{j,\,type \,II}\left(\nu\cdot\int_{Q_j} b_Q(x)dx\,\bar{\nu}\right)\right|
\,\leq\,\frac{3}{4}\sum_{j,\,type\, II}\left|Q_j\right|\leq \frac{3}{4}|Q|\,.$$
Combining our estimates, we obtain
\begin{equation*}
|Q|  \leq I+II\leq C|E|^{\frac{1}{p'}}|Q|^{\frac{1}{p}}+\frac{1}{8}|Q|+\frac{3}{4}|Q|\,,
\end{equation*}
whence it follows that
$|Q|\,\leq\, C |E|\,.$
We now take $0<\eta\leq\frac{1}{C}$, 
so that 
$$\sum_j|Q_j|=|Q\setminus E|=|Q|-|E|\leq (1-\eta)|Q|.$$
Thus, the family 
of cubes that we have constructed satisfies 
condition \eqref{1.cond.1}.  
 
Let us now proceed to verify  (\ref{eq2.17aaa}).
We set
$$E_Q^*:= R_Q\setminus\left(\displaystyle\bigcup_j R_{Q_j}\right)\,.$$ 
We shall prove first that 
\begin{equation}\label{eq2.21a}|z\cdot A_tb_Q(x)\bar{\nu}|\geq\frac{1}{2}|z|\,,
\qquad {\rm if}\,\, z\in\Gamma^{2\epsilon}\,\, {\rm and}\,\,(x,t)\in E_Q^*\,,
\end{equation}
where we recall that $\nu$ is the unit direction vector for $\Gamma^{2\epsilon}$.
Indeed, if $(x,t)\in E_Q^*$, then $Q(x,t)$ is not contained in any selected cube
$Q_j$, and therefore the dyadic average $A_tb_Q$ satisfies the opposite inequalities to those in \eqref{eq2.16a} and \eqref{eq2.17a}.   Thus, by the triangle inequality, for any unit vector 
$\omega\in\mathbb{C}^m$, we have
$$|\omega\cdot A_tb_Q(x)\bar{\nu}| \geq |\nu\cdot A_tb_Q(x)\bar{\nu}|-|(\omega-\nu)A_tb_Q(x)\bar{\nu}|\geq\frac{3}{4}-|(\omega-\nu)|\frac{1}{8\epsilon} \,.$$
If we choose $\omega=\frac{z}{|z|}$, with $z\in \Gamma^{2\epsilon}$, 
then $|\omega-\nu|<2\epsilon$ 
(by definition of $\Gamma^{2\epsilon}$), so that 
$$\left|\frac{z}{|z|}\cdot A_tb_Q(x)\bar{\nu}\right|\geq\frac{3}{4}-\frac{2\epsilon}{8\epsilon}=\frac{1}{2}\,,$$ which yields \eqref{eq2.21a}.
We may then apply \eqref{eq2.21a} with $z=\Theta_t1(x)\in\Gamma^{2\epsilon}$,
to obtain
$$|\Theta_t1(x)|^2\,\leq\, 4\,|\Theta_t1(x)\cdot A_tb_Q(x)\bar{\nu}|^2\,.$$
Since $\bar{\nu}$ is a unit vector,  we obtain \eqref{eq2.17aaa}, and thus also the conclusion of Lemma \ref{l2.6}.  
\end{proof}

\begin{proof}[Verification of Step 1]  
We have already established \eqref{1.cond.1} in Lemma \ref{l2.6}.  It remains to verify
\eqref{1.cond.2}.
To this end, we first 
observe that for $x\in Q,$ and $ \tau_Q(x)\leq t\leq \ell(Q)$,  we have $ (x,t)\in E_Q^*$.
Consequently, by \eqref{eq2.17aaa}, we have
\begin{multline}
\int_Q\left(\int_{\tau_Q(x)}^{\ell(Q)}|\Theta_t1(x)|^2
\mathbbm{1}_{\Gamma^{2\epsilon}}(\Theta1(x))\frac{dxdt}{t}\right)^{\frac{p}{2}}\\
\leq \,C\int_Q\left(\int_0^{\ell(Q)}|\Theta_t1(x)\cdot 
A_tb_Q(x)|^2\frac{dt}{t}\right)^{\frac{p}{2}}dx\,.
\label{2.auxclaimp}\end{multline}
Therefore, to complete the proof of (\ref{1.cond.2}) (and thus also to complete
Step 1), we are reduced to proving that
$$\int_Q\left(\int_{0}^{\ell(Q)}|\Theta_t1(x)\cdot A_tb_Q(x)|^2\frac{dt}{t}\right)^{\frac{p}{2}}dx\leq C|Q|.$$
To this end,  
we use the Coifman-Meyer method and write
$$\Theta_t1A_t=(\Theta_t1)(A_t-P_t)+(\Theta_t1P_t-\Theta_t)+\Theta_t:=R_t^{(1)}
+R_t^{(2)}+\Theta_t\,,$$
where $P_t$ is a nice approximate identity as in Definition \ref{Ptdefinition}.

By (\ref{1.thetabQcondition}), the contribution of $\Theta_tb_Q$ is controlled by $C|Q|$ as desired. Moreover $R_t^{(2)}1=0$, and its kernel satisfies (\ref{1.pointwisebound}) and (\ref{1.pointwisecontinuity}). Thus, by standard Littlewood-Paley/Vector-valued Calder\'{o}n-Zygmund Theory, and condition \eqref{1.pbQcondition}, we have that 
$$\int_Q\left(\int_0^{\ell(Q)}|R_t^{(2)}b_Q(x)|^2\frac{dt}{t}\right)^{\frac{p}{2}}dx
\leq C_p||b_Q||_p^p\leq C|Q|.$$
Furthermore, the same $L^p$ bound holds for $R_t^{(1)}$.  Indeed, since $\theta_t 1$ is uniformly bounded, we may reduce matters to proving the square function bound
$$\int_{\mathbb{R}^n}\left(\int_{0}^{\infty}|(A_t-P_t) f(x)|^2\frac{dt}{t}\right)^{\frac{p}{2}}dx
\lesssim \int_{\mathbb{R}^n}|f(x)|^p dx\,.$$ In turn, one may establish the latter bound by
following the arguments of \cite{DRdeF}.  We omit the details.
\end{proof}

\subsection{Step 2:  Hypotheses of  Lemma \ref{1.lemma} imply hypotheses of  Sublemma \ref{1.sublemma}}\label{c2.s2}

\begin{proof}  Fix a cone $\Gamma^\epsilon$, and define $g_Q$ as in \eqref{eq2.15}.
For a large, but fixed N to be chosen momentarily, let 
$$\Omega_N:=\{x\in Q:g_Q(x)>N\}.$$
Set $E:=Q\setminus(\cup_j Q_j)$, recall that
 $\tau_Q(x):=\sum_j\ell(Q_j)\mathbbm{1}_{Q_j}(x) $,
 and observe that  $\tau_Q\equiv 0$ on $E$.
If the hypotheses of Lemma \ref{1.lemma} hold,
then by Chebyshev's inequality we have
\begin{align*}
|\Omega_N| & \leq\displaystyle\sum_j|Q_j|+|\{x\in E:g_Q(x)>N\}|\\
&\leq(1-\eta)|Q|+\left|\lbrace x\in E:\left(\int_{\tau_Q(x)}^{\ell(Q)}|\Theta_t1(x)|^2 \mathbbm{1}_{\Gamma^{2\epsilon}}(\Theta_t1(x))|\frac{dt}{t}\right)^{\frac{1}{2}}>N\rbrace\right|\\ 
&\leq (1-\eta)|Q|+\frac{1}{N^p}\int_Q\left(\int_{\tau_Q(x)}^{\ell(Q)}|\Theta_t1(x)|^2 \mathbbm{1}_{\Gamma^{2\epsilon}}(\Theta_t1(x))|\frac{dt}{t}\right)^{\frac{p}{2}}dx\\
&\leq(1-\eta)|Q|+\frac{C_1}{N^p}|Q|.
\end{align*}
Choosing N large enough so that $\frac{C_1}{N^p}\leq\frac{\eta}{2}=:\beta$, we obtain 
$ |\Omega_N|\leq(1-\beta)|Q|$.
\end{proof}

\subsection{Step 3:  Proof of Sublemma \ref{1.sublemma}}\label{c2.s3}

\begin{proof}
Let $N,\beta$ be as in the hypotheses. Fix $\gamma\in(0,1)$, a dyadic cube Q, and a cone 
$\Gamma^{\epsilon}$. We first set some notation.  Let
\begin{align}\label{eq2.17aa}
h_{Q,\gamma}(x) &
:=\left(\int_{\gamma}^{min(\ell(Q),\frac{1}{\gamma})}|\Theta_t1(x)|^21_{\Gamma^\epsilon}(\Theta_t1(x))\frac{dt}{t}\right)^{\frac{1}{2}}\,,\\
\label{eq2.17}
g_{Q,\gamma}(x) &
:=\left(\int_{\gamma}^{min(\ell(Q),\frac{1}{\gamma})}|\Theta_t1(x)|^2\chi_{\epsilon}(\Theta_t1(x))\frac{dt}{t}\right)^{\frac{1}{2}}\,,
\end{align} 
where we set these terms to be 0 if $\ell(Q)\leq\gamma$, and where $\chi_\epsilon$ 
is a cut-off function adapted to the cone $\Gamma^{2\epsilon}$, defined as follows.
We let $\chi_\epsilon$ be homogeneous of degree zero in 
$\mathbb{R}^{2m}\cong \CC^m$,  smooth
on the sphere $\mathbb{S}^{2m-1}$, 
with $0\leq \chi_\epsilon\leq 1$, such that $\chi_\epsilon(z)\equiv 1$ on $\Gamma^{\frac32 \epsilon}$,
and is supported on $\Gamma^{2\epsilon}$.
In particular,
\[
  \chi_{\epsilon}(\Theta_t1(x)) = \left\{ 
  \begin{array}{l l l}
    1 & if & \mathbbm{1}_{\Gamma^{\frac{3}{2}\epsilon}}(\Theta_t1(x))=1\\[4pt]
    0 & if & \mathbbm{1}_{\Gamma^{2\epsilon}}(\Theta_t1(x))=0 
  \end{array} \right.
\]
We set
$$k(\gamma):=\sup_Q\frac1{|Q|}\int_Q \left(h_{Q,\gamma}(x)\right)^2 dx \,,
$$
where the supremum runs over all dyadic cubes $Q$. 
We also define 
$$\Omega_{N,\gamma}:=\{x\in Q: g_{Q,\gamma}(x)>N\}\,,$$
which is an open set, by virtue of \eqref{1.pointwisecontinuity} and
the fact that we have made a smooth truncation adapted to the cone $\Gamma^{\epsilon}$.
Note that $k(\gamma)$ is finite for each fixed $\gamma$, and our goal is to show that
 $\sup_{0<\gamma<1}k(\gamma)<\infty$.  We note also that $g_{Q,\gamma}\leq g_Q$, for every $\gamma >0$, where $g_Q$ is defined
 as in \eqref{eq2.15},
and therefore, by \eqref{eq2.14},
\begin{equation}\label{eq2.14aa}
|\Omega_{N,\gamma}|\leq (1-\beta)|Q|\,.
\end{equation}
 
With this notation in place, we begin the proof.  Let
$$F_{N,\gamma}:=Q\setminus\Omega_{N,\gamma}\,,$$ 
and observe that, by \eqref{eq2.14aa}, the set
$F_{N,\gamma}$ is non-empty.
Since $\Omega_{N,\gamma}$ is (relatively) open in $Q$, 
we can make a Whitney decomposition:
$$\Omega_{N,\gamma}=\bigcup_j Q_j\,,$$ 
where the cubes $Q_j$ are a family of non-overlapping dyadic sub-cubes of $Q$,
such that for each $Q_j$ in the decomposition, we have
\begin{equation}\label{eq2.25a}
\dist(Q_j, F_{N,\gamma}) \approx \ell(Q_j)\,.
\end{equation}
We warn the reader (with apologies) about a possible point of confusion:
the present family $\{Q_j\}$ has nothing to do with the family of cubes in 
the statement of Lemma \ref{1.lemma}.

We then have
\begin{multline*}
\int_Q \left(h_{Q,\gamma}(x)\right)^2 dx
=\int_{F_{N,\gamma}} \left(h_{Q,\gamma}(x)\right)^2dx\,\,+\,\,\sum_j\int_{Q_j} 
\left(h_{Q,\gamma}(x)\right)^2dx\\
\leq\int_{F_{N,\gamma}} \left(g_{Q,\gamma}(x)\right)^2dx\,\,+\,\,\sum_j\int_{Q_j} 
\left(h_{Q_j,\gamma}(x)\right)^2dx\\
\qquad+\,\,\sum_j\int_{Q_j}\int_{max(\gamma,\ell(Q_j))}^{min(\ell(Q),\frac{1}{\gamma})}|\Theta_t1(x)|^2\mathbbm{1}_{\Gamma^{\epsilon}}(\Theta_t1(x))\frac{dtdx}{t}\\
\leq N^2|Q|+k(\gamma)\sum_j|Q_j|+\displaystyle\sum_j\int_{Q_j}\int_{max(\gamma,\ell(Q_j))}^{min(\ell(Q),\frac{1}{\gamma})}|\Theta_t1(x)|^2\mathbbm{1}_{\Gamma^{\epsilon}}(\Theta_t1(x))\frac{dtdx}{t}\\
\leq N^2|Q|+k(\gamma)(1-\beta)|Q|+\displaystyle\sum_j\int_{Q_j}\int_{max(\gamma,\ell(Q_j))}^{min(\ell(Q),\frac{1}{\gamma})}|\Theta_t1(x)|^2\mathbbm{1}_{\Gamma^{\epsilon}}(\Theta_t1(x))\frac{dtdx}{t}\,.
\end{multline*}
We now make the following
\begin{claim}
\begin{equation}
L_j:=\int_{Q_j}\int_{max(\ell(Q_j),\gamma)}^{min(\ell(Q),\frac{1}{\gamma})}|\Theta_t1(x)|^2\mathbbm{1}_{\Gamma^{\epsilon}}(\Theta_t1(x))\frac{dtdx}{t}\leq C|Q_j|.
\end{equation}
\label{1.claim}\end{claim}
Given the claim, we may divide by $|Q|$, and then take a supremum in $Q$, to obtain the uniform bound
$$ k(\gamma) \leq \frac{C_N}{\beta}\,.$$
Therefore letting $\gamma$ approach zero we have that
$$\int_Q\int_0^{\ell(Q)}|\Theta1(x)|^21_{\Gamma^{\epsilon}}(\Theta_t1(x))\frac{dtdx}{t}\leq C|Q|\,,$$
uniformly for all cubes $Q$ and all cones $\Gamma^\epsilon$.
Summing over the cones $\Gamma^\epsilon_k$, we  conclude that
\begin{align*}
\int_Q\int_0^{\ell(Q)}|\Theta_t1(x)|^2\frac{dt}{t}dx &
\leq\int_Q\int_0^{\ell(Q)}\sum_k|\Theta_t1(x)|^2\mathbbm{1}_{\Gamma_k^{\epsilon}}(\Theta_t1(x))\frac{dt}{t}dx\\ 
&\leq\sum_k\int_Q\int_0^{\ell(Q)}|\Theta_t1(x)|^2\mathbbm{1}_{\Gamma_k^{\epsilon}}(\Theta_t1(x))\frac{dt}{t}dx\\    
&\leq \,C K(\epsilon,m) \, |Q|,
\end{align*}
where $K(\epsilon, m)$ is the number of cones of aperture $\epsilon$ needed 
to cover $\mathbb{C}^m
\cong \mathbb{R}^{2m}$.
\end{proof}

\begin{proof}[Proof of Claim \ref{1.claim}]

Let $0<\beta<\alpha$, where $\alpha$ is the exponent  in
the kernel estimates  \eqref{1.pointwisebound} and \eqref{1.pointwisecontinuity}, 
and define the following sets
\begin{align*}
Q_j^{(1)} & :=\left\{x\in Q_j: \ |\Theta_t1(x)|\leq\left(\frac{\ell(Q_j)}{t}\right)^{\beta}\frac{1}{\epsilon}\right\};\\
Q_j^{(2)} & :=\{x\in Q_j: \ \mathbbm{1}_{\Gamma^{\epsilon}}(\Theta_t1(x))=0\};\\
Q_j^{(3)} & :=Q_j\setminus( Q_j^{(1)}\cup Q_j^{(2)}).\\
\end{align*}
Then $L_j\leq L_j^1+L_j^2+L_j^3$ where 
$$L_j^i=\int_{Q_j^{(i)}}\int_{max(\ell(Q_j),\gamma)}^{min(\ell(Q),\frac{1}{\gamma})}|\Theta_t1(x)|^2\mathbbm{1}_{\Gamma^{\epsilon}}(\Theta_t1(x))\frac{dtdx}{t}\,, \qquad  i=1,2,3\,.$$
Trivially, $ L_j^2=0$.   Moreover,  
\begin{equation*}
 L_j^1 
\,\leq\,  \int_{Q_j}\int_{\ell(Q_j)}^{\ell(Q)}\left(\frac{\ell(Q_j)}{t}\right)^{\beta}\frac{1}{\epsilon}\frac{dxdt}{t}
\,\leq \, C(n,\beta,\epsilon)\,|Q_j|.
\end{equation*}
To estimate $L_j^3$, we note that by the standard property of Whitney cubes (i.e., \eqref{eq2.25a},
to be precise), there is a point $ x_j\in F_{N,\gamma}$ 
such that $\dist(x_j,Q_j)\lesssim \ell(Q_j)$.  We fix such an $x_j$, 
and a sufficiently large dimensional constant $C_n$ to be chosen,
and decompose $L_3$ as follows:
\begin{multline*}
 L_j^3\lesssim\int_{Q_j^{(3)}}\int_{\ell(Q_j)}^{C_n\ell(Q_j)}|
\Theta_t1(x)|^2\mathbbm{1}_{\Gamma^{\epsilon}}(\Theta_t1(x))\frac{dtdx}{t}\\+\int_{Q_j^{(3)}}\int_{C_nl(Q_j)}^{\ell(Q)}|\Theta_t1(x)\mathbbm{1}_{\Gamma^{\epsilon}}(\Theta_t1(x))-\Theta_t1(x_j)\chi_{\epsilon}(\Theta_t1(x_j))|^2\frac{dtdx}{t}\\
+\int_{Q_j^{(3)}}\int_{\gamma}^{min(\ell(Q),\frac{1}{\gamma})}
|\Theta_t1(x_j)|^2\,\chi_{\epsilon}\left(\Theta_t1(x_j)\right)\frac{dtdx}{t}\\
=: I+II+III\,.
\end{multline*}
For I,  we use that $\sup_{t>0}\Vert \Theta_t 1\Vert_{\infty}<\infty$, and for III,
that $x_j\in F_{N,\gamma}$, so $I+III\leq C(n,N)|Q_j|$.

For II we have two cases:  

\noindent\underline{Case 1}: $\mathbbm{1}_{\Gamma^{\epsilon}}(\Theta_t1(x_j))=1$.
Then for $C_n$ large enough, by \eqref{1.pointwisecontinuity}, we have 
\begin{equation}\label{eq2.28a}
|\Theta_t1(x)-\Theta_t1(x_j)|\leq C\left(\frac{\ell(Q_j)}{t}\right)^{\alpha}\,,\qquad \forall \, x\in Q_j \,.
\end{equation}
Consequently,
 $II\leq\int_{Q_j}\int_{C_n \ell(Q_j)}^{\infty}\left(\frac{\ell(Q_j)}{t}\right)^{2\alpha}\frac{dtdx}{t}\leq C|Q_j|.$

\noindent\underline{Case 2}: $\mathbbm{1}_{\Gamma^{\epsilon}}(\Theta_t1(x_j))=0$.  Then
$|\nu-\frac{\Theta_t1(x_j)}{|\Theta_t1(x_j)|}|>\epsilon$.  On the other hand, for $x\in Q_j^{(3)}$,
we have that $|\nu-\frac{\Theta_t1(x)}{|\Theta_t1(x)|}|\leq\epsilon$, and also that
 $|\Theta_t1(x)|>\left(\frac{\ell(Q_j)}{t}\right)^{\beta}\frac{1}{\epsilon}$, or equivalently, that  
$ \frac{1}{|\Theta_t1(x)|}<\left(\frac{t}{\ell(Q_j)}\right)^{\beta}\epsilon$.  
 Thus, using \eqref{eq2.28a}, and the elementary inequality in Remark \ref{remark2.8} below, we have
 for $C_n$  large enough, that
\begin{multline*}\left|\frac{\Theta_t1(x)}{|\Theta_t1(x)|}-\frac{\Theta_t1(x_j)}{|\Theta_t1(x_j)|}\right|\leq 2|\Theta_t1(x)-\Theta_t1(x_j)|\cdot\frac{1}{|\Theta_t1(x)|}\\\leq (2C)\left(\frac{\ell(Q_j)}{t}\right)^{\alpha-\beta}\epsilon\leq C\left(\frac{1}{C_n}\right)^{\alpha-\beta}\epsilon\leq\frac{\epsilon}{2},\end{multline*}
so that
\begin{equation*}\left|\nu-\frac{\Theta_t1(x_j)}{|\Theta_t1(x_j)|}\right|\leq\left|\frac{\Theta_t1(x)}{|\Theta_t1(x)|}-\frac{\Theta_t1(x_j)}{|\Theta_t1(x_j)|}\right|+\left|\nu-\frac{\Theta_t1(x)}{|\Theta_t1(x)|}\right| \leq\frac{\epsilon}{2}+\epsilon\leq\frac{3}{2}\epsilon\,.
\end{equation*}
By the definition of the cut-off $\chi_\epsilon$, it then follows that
\begin{equation*}
 II\leq\int_{Q_j}\int_{C_n\ell(Q_j)}^{\ell(Q)}|\Theta_t1(x)-\Theta_t1(x_j)|
 \frac{dtdx}{t}\leq C|Q_j|\,,
\end{equation*} where again we have used \eqref{eq2.28a} as in Case 1.

\end{proof}

\begin{remark}\label{remark2.8}  Observe that
$$\left|(x|y|)-(y|x|)\right|\leq \left|(x|y|)-(y|y|)\right|+\left|(y|y|)-(y|x|)\right|\leq 2|y|\cdot|x-y|,$$
so that
$$\frac{\left|x|y|-y|x|\right|}{|x||y|}\leq 2\frac{|x-y|}{|x|}
\,\,\,\Rightarrow\left|\frac{x}{|x|}-\frac{y}{|y|}\right|\leq 2\frac{|x-y|}{|x|}.$$
\end{remark}

\section{Local Tb Theorem without pointwise kernel bounds, version 1}\label{c4}

In this section, we replace the pointwise conditions on the kernel of $\Theta_t$, by
$L^2$ or $L^q$ conditions.  For testing functions
in $L^p$, with $p<2$, this appears to require that 
we work with conical (local) square functions rather than vertical ones.  

We begin by introducing some auxiliary operators.
\begin{definition}\label{3.lpfamily}  We say that a family of convolution operators
$\{Q_s\}_{s>0}$ is a {\bf CLP family} (``Calder\'on-Littlewood-Paley" family)  if
for some $\sigma >0$, and some
$\psi\in L^1(\RR^n)$, with $|\psi(x)|\lesssim (1+|x|)^{-n-\sigma}$, 
we have
\begin{align*} Q_s f = \,s^{-n}\psi(\cdot/s) * f\,,
 \,{\rm and}\,\,
\widehat{\psi}(\xi)&\leq  \,C\min \left( |\xi|^\sigma,|\xi|^{-\sigma}\right)\\[4pt]
\sup_{s>0}\left( \| Q_sf\|_{L^2(\mathbb{R}^n)}
+||s\nabla Q_sf||_{L^2(\mathbb{R}^n)}\right)&\leq C\|f\|_{L^2(\mathbb{R}^n)}\,,\\[4pt]
\int_{\mathbb{R}^n}\int_0^{\infty}|Q_sf(x)|^2\frac{dsdx}{s}&\leq C||f||_{L^2(\mathbb{R}^n)}^2\,, 
\\[4pt]
\int_0^{\infty}Q_s^2\frac{ds}{s}&=I\,,
\end{align*}
where convergence to the identity in the last formula is in the strong 
operator topology on $\mathcal{B}(L^2)$. 
\end{definition}

We also introduce some additional notation.  
\begin{itemize}
\item Given a dyadic cube $Q$, we let $\DD(Q)$ denote the collection of all
dyadic sub-cubes of $Q$ (including, of course, $Q$ itself).
\item Given a cube $Q$, we let $U_Q:= Q\times (\ell(Q)/2,\ell(Q)]$ be the
``Whitney box" above $Q$.
\item Given $x\in \RR^n$, we define the ``dyadic cone" with vertex at $x$ by
\begin{equation}\label{eq3.1}
\GD(x):= \bigcup_{Q\in \DD:\, x\in Q} U_Q\,,
\end{equation}
and given a dyadic cube $Q$, for $x\in Q$,
we define the ``truncated dyadic cone" with vertex at $x$, of height $\ell(Q)$, by
\begin{equation}\label{eq3.2}
\GD_Q(x):= \bigcup_{Q'\in \DD(Q):\, x\in Q'} U_{Q'}
\end{equation}
\end{itemize}

\begin{definition}\label{3.generalsqfunc}

 We consider a family of operators $\{\Theta_t\}_{t>0}$, taking values in $\CC^m,\,m\geq 1$,
 so that  $\Theta_t := (\Theta_t^{1},\Theta_t^{2},
...,\Theta_t^{m})$, and
for $f=(f^1,f^2,...,f^m)\in L^2(\mathbb{R}^n,\CC^m)$, we set
$$ \Theta_tf:=\sum_{j=1}^m \Theta_t^{j}f^j\,.$$ 
We also define the action of $\Theta_t$ on an $m\times m$ matrix valued function $b =(b_{j,k})$,
in the obvious way, i.e., $\Theta_t b=((\Theta_t b)_1,(\Theta_t b)_2,...,(\Theta_t b)_m)$
is a $\CC^m$ valued function, with
$$\left(\Theta_t b\right)_k:= \sum_{j=1}^m \Theta^j_t b_{j,k}\,.$$
We suppose that $\Theta_t$ satisfies the following properties: 

\noindent
\item[(a)] (Uniform $L^2$ bounds and off-diagonal decay in $L^2$).
\begin{equation}
\displaystyle \sup_{t>0}||\Theta_t f||_{L^2(\mathbb{R}^n)}\leq C||f||_{L^2(\mathbb{R}^n)}\,,
\label{3.supl2bound}\end{equation}

\begin{equation}
||\Theta_t f_j||_{L^2(Q)}\leq C2^{-j(n+2+\beta)/2}||f_j||_{L^2(2^{j+1}Q\setminus 2^jQ)}\,, \quad 
\ell(Q)\leq t\leq 2\ell(Q)\,,
\label{3.annulibound}\end{equation} 
\noindent for some $\beta>0$, where $f_j:=f\mathbbm{1}_{2^{j+1}Q\setminus 2^jQ}$.

\smallskip

\noindent(b) (Quasi-orthogonality in $L^2$).   
For some (hence every) CLP family $\{Q_s\}$, 
there is a $\beta>0$ for which we have
\begin{equation}
||\Theta_t Q_sh||_{L^2(\mathbb{R}^n)}\leq C\left(\frac{s}{t}\right)^{\beta}||h||_{L^2(\mathbb{R}^n)}, \quad 
 \forall s\leq t\,.
\label{3.Qsbound}\end{equation}

\smallskip

\noindent(c) (``Hypercontractive" off-diagonal decay). There is some $1<r<2$,  and some
$\mu>\frac{n}{r}$ $(\mu=\frac{n}{r}+\varepsilon,$ for some $\varepsilon>0)$, such that
\begin{multline}
\left(\int_{Q^*}|\Theta_t(f\mathbbm{1}_{S_j}(Q))(y)|^2dy\right)^{\frac{1}{2}}\leq C2^{-j\mu}t^{-n(\frac{1}{r}-\frac{1}{2})}
\left(\int_{S_j(Q)}|f(y)|^rdy\right)^{\frac{1}{r}}\,,\\[4pt] \forall j \geq 0\,, \,\ell(Q)<t\leq2\ell(Q)\,,
\label{3.rcondition}\end{multline}
where $S_0(Q)=16Q$ and $S_j(Q)=2^{j+4}Q\setminus2^{j+3}Q$, $j\geq 1$.
$Q^*\equiv 8Q$.

\smallskip

\noindent(d) (Improved integrability). There is an exponent $q>2$ such that
\begin{equation}  
\sup_{t>0}||\Theta_tf||_{L^q(\mathbb{R}^n)}\leq C||f||_{L^q(\mathbb{R}^n)}\,.
\label{3.qcondition}\end{equation}
\end{definition}

\begin{remark}\label{remark 4.1b} If \eqref{3.annulibound} 
holds for all $t>0$ (not just for $t\approx \ell(Q)$), as is often  the case in applications,
then hypothesis (d) is redundant, and in fact \eqref{3.qcondition} holds for all $q>2$;
indeed, one has the pointwise bound
$$\mathcal{M}(\Theta_t f) \lesssim \left(\mathcal{M}(|f|^2)\right)^{1/2}\,.$$
We leave the details to the interested reader.
\end{remark}

\begin{remark}\label{remark 4.1a}  We observe that, for example, (b), (c), and (d) hold,
with $\theta_t=t\partial_t P_t,$ where $P_t=e^{-t\sqrt{-\Delta}}$ is the usual Poisson semigroup,
and that  (a) holds with $\beta=0$, for the same operator.   We may obtain a positive value of $\beta$ in (a), by considering higher order derivatives of $P_t$.   As a practical matter, when considering
square functions arising in PDE applications, it is often a fairly routine matter to pass to higher order
derivatives.
The advantage of the present formulation of
our conditions, is that these conditions may continue to hold in the absence of pointwise kernel bounds. 
In PDE applications, (d) is typically obtained as a consequence of higher integrability estimates
of ``N. Meyers" type (cf. \cite{Me2}).
\end{remark}

 We shall need to work with some ``dyadic conical" square functions, and their local analogues. 

 \begin{definition}\label{d3.10}
Given $F$ with domain $\RR^{n+1}_+$, we define the 
``dyadic conical square function" of $F$ by
$$\AD (F)(x):=\left(\iint_{\GD(x)}|F(y,t)|^2\frac{dy dt}{t^{n+1}}\right)^{1/2}\,.$$
Similarly, we define the
``truncated dyadic conical square function" relative to a dyadic cube $Q$ by
$$\ADQ (F)(x):=\left(\iint_{\GD_Q(x)}|F(y,t)|^2\frac{dy dt}{t^{n+1}}\right)^{1/2}\,.$$
\end{definition} 

\begin{remark}\label{r3.3}  We observe that $\ADQ (F)$ 
is dominated by the standard local square function
$$\AC^\alpha_Q (F)(x):= \left(\int_0^{\ell(Q)}\!\!\int_{|x-y|<\alpha t}|F(y,t)|^2
\frac{dydt}{t^{n+1}}\right)^{1/2}\,,$$
assuming that the aperture $\alpha$ is chosen large enough, as may be seen from the definition of the 
dyadic cones, and a trivial geometric argument.
\end{remark} 

As before, for
$\epsilon$ small but fixed, we cover $\mathbb{C}^m$ by cones $\Gamma^\epsilon_k$,
of aperture $\epsilon$, with vertex at the origin. 
The constants in our estimates 
are then allowed to depend on $K=K(\epsilon,m)$, the number of cones
in the covering.  Given any such cone $\Gamma^\epsilon$,
we shall also need to consider dyadic conical square functions ``restricted" to the doubled cone
$\Gamma^{2\epsilon}$:
 \begin{equation}\label{eq3.11}
 \ADQG (F)(x):=\left(\iint_{\GD_Q(x)}|F(y,t)|^2\,
  \mathbbm{1}_{\Gamma^{\epsilon}}\big(F(y,t)\big)
\frac{dy dt}{t^{n+1}}\right)^{1/2}\,.
 \end{equation}

Our main result in this section is the following:  
\begin{theorem}\label{3.genmatrixtheorem}
Let $\{\Theta_t\}_{t>0}$, be as in Definition \ref{3.generalsqfunc} above,
 and suppose that there exist  positive constants $C_0<\infty$, and $\delta>0$, 
 an exponent $p>r$,  and a system $\{b_Q\}$ of complex $m\times m$ matrix-valued
 functions indexed by  dyadic
cubes $Q\subset\mathbb{R}^n$, 
such that for each dyadic cube Q:

\begin{equation}
\int_{\mathbb{R}^n}|b_Q(x)|^pdx\leq C_0|Q|,
\label{3.pbQcondition}\end{equation}

\begin{equation}
\int_Q\left(\ADQ (\Theta_t b_Q)(x)\right)^{p}dx\leq C_0|Q|,
\label{3.thetabQcondition}\end{equation}


\begin{equation}
\delta\,|\xi|^2 
\,\leq\, Re\left(\xi\cdot\fint_Qb_Q(x) dx 
\, \bar{\xi}\right), \qquad \forall\xi\in\mathbb{C}^m \,.
\label{3.ellipticitybQcond}\end{equation}
Then
\begin{equation}
\iint_{\mathbb{R}_+^{n+1}}|\Theta_tf(x)|^2\frac{dxdt}{t}\leq C||f||_{L^2(\mathbb{R}^n)}^2, 
\label{3.conclusion}\end{equation}

\label{3.maintheorem}\end{theorem}


\begin{remark}\label{r3.6}
We observe that by Remark \ref{r3.3},
we could replace $\ADQ$ by $\AC^\alpha_Q$, with $\alpha$ sufficiently large,
in hypothesis \eqref{3.thetabQcondition}.
\end{remark}

We begin with a generalization of the Christ-Journ\'e $T1$ theorem for square functions.

\begin{theorem}(T1 Theorem)
Suppose that $\Theta_tf(x)$ satisfies conditions \eqref{3.supl2bound}, \eqref{3.annulibound} 
and \eqref{3.Qsbound} as above, 
and also the Carleson measure estimate
\begin{equation}
\displaystyle\sup_Q\frac{1}{|Q|}\int_Q\int_0^{\ell(Q)}|\Theta_t1(x)|^2\frac{dxdt}{t}\leq C.
\label{3.carlessoncondition}\end{equation}
Then we have the square function estimate
\begin{equation}
\iint_{\mathbb{R}_+^{n+1}}|\Theta_tf(x)|^2\frac{dxdt}{t}\leq C||f||_{L^2(\mathbb{R}^n)}^2\,.
\end{equation}
\label{3.T1theorem}\end{theorem}

 \begin{remark} Here, the constant function $1$ should be interpreted in the matrix-valued sense,
i.e., as the $m\times m$ identity matrix.
\label{remark4.4}\end{remark}

We defer the proof  of Theorem \ref{3.T1theorem} to an appendix (Section \ref{st1}).



\begin{lemma}\label{l3.7}
Suppose that there exists $\eta\in(0,1)$, $\epsilon>0$ small and $C<\infty$ such that for every dyadic cube $Q\in\mathbb{R}^n$, and
for each cone $\Gamma^{\epsilon}$ of aperture $\epsilon$,
there is a family $\{Q_j\}_j$ of non-overlapping  dyadic sub-cubes of $Q$,  
with \begin{equation}\label{eq3.18}
\sum_j|Q_j|\leq(1-\eta)\left|Q\right|
\end{equation}
and
\begin{equation}
\int_E\left(\ADQG(\Theta_t 1)(x)\right)^{p}dx\,\leq\, C\,|Q|
\label{3.condlemma}\end{equation}
where $E:=Q\setminus\{\cup_j Q_j\}$.
Then \begin{equation*}\displaystyle\sup_Q\frac{1}{|Q|}\int_0^{\ell(Q)}\!\!\int_Q|\Theta_t1(x)|^2\frac{dxdt}{t}\leq C\end{equation*}
\label{3.lemma}\end{lemma}

\begin{sublemma}

Suppose that there exist $N<+\infty$,  and $\beta\in(0,1)$, such that for every cube Q and for each cone $\Gamma^{\epsilon}$
\begin{equation}
|\{x\in Q: G_Q(x)>N\}|\leq(1-\beta)\left|Q\right|\,,
\end{equation}
where
\begin{equation}
G_Q:= 
\ADQG(\Theta_t 1) \,. 
\end{equation}
Then
\begin{equation}
\displaystyle\sup_Q\frac{1}{|Q|}\int_0^{\ell(Q)}\!\!\int_Q|\Theta_t1(x)|^2\frac{dxdt}{t}\leq C.
\label{3.conclusionsublemma}\end{equation}
\label{3.sublemma}\end{sublemma}

The scheme of our proof now follows the  three-step argument of Section \ref{c2}.

\subsection{Step 1:  Hypotheses of Theorem \ref{3.maintheorem} 
imply hypotheses of Lemma \ref{3.lemma}}\label{c4.s1}

\begin{proof}  We fix a dyadic cube $Q$,  and a cone 
$$\Gamma^{\epsilon}=\left\{z\in\mathbb{C}^m:\left|\frac{z}{|z|}-\nu\right|<\epsilon\right\}\,,$$
with aperture $\epsilon$, and unit direction vector
$\nu\in\mathbb{C}^m$.  
As above, we define the usual dyadic averaging operator, by setting
$A_{t}f(x):=|Q(x,t)|^{-1}\int_{Q(x,t)}f(y)dy$, where $Q(x,t)$ is the 
 minimal dyadic sub-cube of $Q$, containing x, with side length at least t.

Assuming the hypotheses of Theorem \ref{3.genmatrixtheorem}, we then obtain by Lemma
\ref{l2.6} above, that for $\epsilon$ small enough, we can
construct the required family $\{Q_j\}$ of dyadic sub-cubes of $Q$, satisfying
 \eqref{eq3.18}, as 
well as \begin{equation}\label{eq3.23}
|\Theta_t1(x)|^2\mathbbm{1}_{\Gamma^{\epsilon}}\big(\Theta_t1(x)\big)\leq 
4|\Theta_t1(x) A_{t}b_Q(x)|^2\,,\qquad \forall (x,t) \in E_Q^*:= R_Q\setminus
\left(\bigcup_j R_{Q_j}\right)\,.
\end{equation}

With \eqref{eq3.23} in hand, we may now verify the second condition of the lemma, \eqref{3.condlemma}.  We first note that for $x\in E =Q\setminus (\cup_j Q_j)$, 
if $x\in Q'\in \DD(Q)$,  then $Q'$ is not contained in any $Q_j$.  
Thus, for $x\in E$, by definition of the dyadic cones, we have
$\GD_Q(x)\subset E_Q^*$. 
Consequently, by \eqref{eq3.23},
\begin{equation*}\int_E\left(\ADQG (\Theta_t 1)\right)^p dx \lesssim
\int_Q \left( 
\int\!\!\!\int_{\GD_Q(x)}|\Theta_t1(y) A_{t}b_Q(y)|^2
\frac{dydt}{t^{n+1}}\right)^{p/2} dx\,.
\end{equation*}
We now use the Coifman-Meyer trick to write
\begin{multline}\label{eq3.remainderdef}\Theta_t1 A_{t}b_Q = \big((\Theta_t1) A_{t}-(\Theta_t1)A_tP_t\big) b_Q\,+
 \big((\Theta_t1) A_{t}P_t-\Theta_t\big) b_Q\,
+\, \Theta_t b_Q\\[4pt]=: R^{(1)}_t b_Q \,+\,R_t^{(2)}b_Q\,+\,\Theta_tb_Q\,,
\end{multline}
where as usual $P_t$ denotes a nice approximate identity operator
(cf. Definition \ref{Ptdefinition} above).
By \eqref{3.thetabQcondition},
the contribution of the term $\Theta_tb_Q$ gives the desired bound.

To treat the two ``remainder terms" $R_t^{(1)}b_Q$ and $R_t^{(2)}b_Q$, 
we first recall the following. 
\begin{proposition}\cite{C-UMP}
Let $T$ be a sublinear operator satisfying
$$T:L^2(\mathbb{R}^n,v)\rightarrow L^2(\mathbb{R}^n,v)\,, \qquad 
\forall v\in A_{2/r}\,.$$ Then
$T:L^p(\mathbb{R}^n)\rightarrow L^p(\mathbb{R}^n), \ for \ p>r.$
\end{proposition}
By the Proposition, and \eqref{3.pbQcondition}, we are left to prove, for $i=1,2$, that
\begin{equation}\label{eq3.24}
\int_{\mathbb{R}^n}|\AD(R_t^{(i)}f)(x)|^2v(x)dx\leq C\int_{\mathbb{R}^n}|f(x)|^2v(x)dx\,, 
\qquad  v\in A_{2/r}\,.
\end{equation}
By definition of the dyadic cones used to construct $\AD$, the left side of \eqref{eq3.24} equals
$$\int_{\RR^n}\sum_{Q\in\DD}1_Q(x)\int\!\!\!\int_{U_Q} |R_t^{(i)}f(y)|^2\frac{dydt}{t^{n+1}}\,v(x) \,dx
\,\leq\,\int_0^\infty\!\!\int_{\RR^n}t^{-n}\!\int_{|x-y|<Ct} |R_t^{(i)}f(y)|^2 dy\, v(x)\,dx \,\frac{dt}{t}$$
Thus, by a standard orthogonality argument, it suffices to prove that for some $\beta_1>0$, we have
\begin{equation}\label{eq3.25}
\int_{\mathbb{R}^n}t^{-n}\!\int_{|x-y|<Ct}|R^{(i)}_tQ_sh(y)|^2dy\,v(x)dx\leq C \min\left(\frac{s}{t},\frac{t}{s}\right)^{\beta_1}\int_{\mathbb{R}^n}|h(x)|^2v(x)dx\,,
\end{equation}
where $\{Q_s\}_{s>0}$ is a CLP family as in Definition \ref{3.lpfamily}, which in addition satisfies
the weighted $L^2$ estimate
$$\int_0^\infty\!\!\int_{\RR^n} |Q_s f(x)|^2 v(x)\frac{dxdt}{t}\,
\leq\, C\int_{\mathbb{R}^n}|f(x)|^2v(x)\,dx\,, \qquad v\in A_{2/r}\,.$$
To this end, we begin by showing that \eqref{eq3.25} holds in the unweighted case $v\equiv 1$,
i.e., that for some $\beta_0>0$, we have
\begin{equation}\label{eq3.26}
\int_{\mathbb{R}^n}t^{-n}\!\int_{|x-y|<Ct}|R^{(i)}_tQ_sh(y)|^2dy\,dx\,\approx
\int_{\mathbb{R}^n}|R^{(i)}_tQ_sh(y)|^2dy 
\lesssim \min\left(\frac{s}{t},\frac{t}{s}\right)^{\beta_0}\int_{\mathbb{R}^n}|h(x)|^2dx\,.
\end{equation}
In fact, more generally, we shall prove the following
\begin{lemma}\label{l3.10}  Let $H$ be a subspace of
$L^2(\rn,\CC^m),\, m\geq 1$, and suppose that $\Theta_t$ satisfies \eqref{3.supl2bound},
\eqref{3.annulibound}, and, for every $h\in H$, \eqref{3.Qsbound}.  Let $R_t^{(i)},\, i=1,2$,
be defined as in \eqref{eq3.remainderdef}.  Then there is some $\beta_0>0$, such that 
the bound \eqref{eq3.26} holds for all $h\in H$, and for $i=1,2$.  
\end{lemma}
We remark that for our purposes at present, we simply take $H$ to be all of $L^2(\rn,\CC^m)$,
but we record the more general version stated above, as well as the following
Corollary, for future reference.  
\begin{corollary}\label{c3.11}  Under the hypotheses of Lemma \ref{l3.10}, for $i=1,2$,
we have the square function bound
$$\iint_{\uh}|R^{(i)}_t h(x)|^2 \frac{dxdt}{t} \leq C \|h\|_2^2\,,\qquad \forall h\in H\,.$$
\end{corollary}
The Corollary follows from the Lemma by a standard orthogonality
argument of \  ``Schur's lemma" type.  We omit the routine details.

\begin{proof}[Proof of Lemma \ref{l3.10}]
We note that by  \eqref{3.supl2bound} and
\eqref{3.annulibound}, we may invoke
\cite[Lemma 3.11]{AAAHK} (cf. Lemma \ref{lemma3.11AAAHK} above), to deduce that
$\Theta_t 1$ is well defined, and satisfies
\begin{equation}\label{eq3.27}
\sup_{t>0}\|(\Theta_t1) A_t\|_{L^2(\mathbb{R}^n)\to L^2(\mathbb{R}^n)}
 \leq C\,,
\end{equation}  
and therefore also that each $R_t^{(i)}$
satisfies \eqref{3.supl2bound} and
\eqref{3.annulibound}.  Thus, since in addition,
we have that
$R_t^{(i)} 1=0$ for $i=1,2$, we may  invoke \cite[Lemma 3.5]{AAAHK}, to 
obtain\footnote{\cite[Lemma 3.11 and Lemma 3.5]{AAAHK}
are stated for operators satisfying our off-diagonal decay estimate \eqref{3.annulibound}
for all $t\lesssim\ell(Q)$, but the proof in \cite{AAAHK} actually requires only that off-diagonal decay 
bounds hold with $t\approx \ell(Q)$.  Also \cite[Lemma 3.5]{AAAHK} is stated under the hypothesis
that \eqref{3.annulibound} holds with $\beta \geq 2$, but the proof given there actually requires only that $\beta >0$, as we assume here.}  that
$$\|R_t^{(i)}Q_s h\|_2\lesssim t \ \|\nabla Q_s h\|_2\,,\qquad \forall h\in L^2\,,$$
which yields \eqref{eq3.26}, in the case $t\leq s$, with $\beta_0 = 2$, by properties of $Q_s$ 
(cf. Definition \ref{3.lpfamily}).

Consider now the case $s\leq t$ of \eqref{eq3.26}.
Since $A_t$ is a projection operator,
we have
$$R_t^{(1)} = (\Theta_t 1) A_t (A_t-P_t)\,,$$
so that the desired bound for this term, for all $h\in L^2$, follows from
\eqref{eq3.27}, and the well known fact that there is some $\mu>0$ such that
$$\|A_t Q_s\|_{L^2\to L^2} \,+ \,\|P_t Q_s\|_{L^2\to L^2} \,\lesssim \left(\frac{s}{t}\right)^\mu\,,
\qquad s\leq t\,.$$
Similarly, the case $s\leq t$ of \eqref{eq3.26} follows for $R_t^{(2)}$ from the latter fact, and 
the fact that \eqref{3.Qsbound} holds for all $h\in H$.  
\end{proof}

With \eqref{eq3.26} in hand, we return to the proof of \eqref{eq3.25}, and we 
make the following claim:
\begin{claim}
\begin{equation}\label{eq3.28}
\int_{\mathbb{R}^n}t^{-n}\!\int_{|x-y|<Ct}|R^{(i)}_tQ_sh(y)|^2dy\,\tilde{v}(x)dx\leq C \int_{\mathbb{R}^n}|h(x)|^2\tilde{v}(x)dx,\quad \forall \tilde{v}\in A_{2/r}\,,
\end{equation}
with $i=1,2$.
\label{2.claim2}\end{claim}

Interpolating with change of measure (\cite{SW}) between \eqref{eq3.26} and \eqref{eq3.28},
we get \eqref{eq3.25}. Indeed, 
for each $v\in A_{\frac{2}{r}}$, there exist $\tau>0$ such that $v^{1+\tau}\in A_{\frac{2}{r}}$, 
so we choose $\tilde{v}(x)=v^{1+\tau}(x)$.
To our knowledge, the idea of using interpolation with change of measure in this way first appeared in the paper \cite{DRdeF}.
Modulo the proof of the claim, Step 1 is now complete.
\end{proof}

\begin{proof}[Proof of Claim \ref{2.claim2}]

Define $\tilde{h}(x)=Q_sh(x)$. By  properties of $Q_s$ (it is controlled by the maximal operator), 
and since $A_{2/r}\subset A_2$, we have $||\tilde{h}||_{L^2_{\tilde{v}}(\mathbb{R}^n)}\leq C||h||_{L^2_{\tilde{v}}(\mathbb{R}^n)}$. 
Thus, it is enough to prove
\begin{equation*}
\int_{\mathbb{R}^n}t^{-n}\!\int_{|x-y|<Ct}|R_t\tilde{h}(y)|^2dy\,\tilde{v}(x)dx
\,\leq\, C\int_{\mathbb{R}^n}\tilde{h}(x)^2\tilde{v}(x)dx.
\end{equation*}

For $t>0$, let $\DD(t)$ denote the grid of dyadic cubes $P\subset \RR^n$ with length 
$\ell(P)\in (t/2,t]$, and let $P^*$ denote the concentric dilate $P^*= \kappa P$ for some large $\kappa$.
Recall that $S_j(P):= 2^{j+4}P\setminus 2^{j+3} P$.
By property (c) of Definition \ref{3.generalsqfunc}, and since $\tilde{v}\in A_{2/r}$ (which gives us 
an $L^{2/r}_{\tilde{v}}$ bound for the maximal function), we have
\begin{multline}\label{eq3.22}
\left(\int_{\mathbb{R}^n}t^{-n}
\!\int_{|x-y|<Ct}|R_t^{(i)}\tilde{h}(y)|^2dy\,\tilde{v}(x)dx\right)^{1/2} 
=\left(\displaystyle\sum_{P\in\mathbb{D}(t)}\int_Pt^{-n}\!\int_{|x-y|<Ct}|R_t^{(i)}\tilde{h}(y)|^2dy\,\tilde{v}(x)dx\right)^{1/2}
\\[4pt]\leq\, C\left(\displaystyle\sum_{P\in\mathbb{D}(t)}\int_P\frac1{|P^*|}\!\int_{P^*}|R^{(i)}_t
\tilde{h}(y)|^2dy\,\tilde{v}(x)dx\right)^{1/2}\\[4pt]
\leq C\displaystyle\sum_{j=0}^{\infty}\left(\displaystyle\sum_{P\in\mathbb{D}(t)}\int_P
\frac1{|P^*|}\!\int_{P^*}|R^{(i)}_t(\tilde{h}\mathbbm{1}_{S_j(P)})(y)|^2dy\tilde{v}(x)dx\right)^{1/2}\\
\leq C\displaystyle\sum_{j=0}^{\infty}\left(\displaystyle\sum_{P\in\mathbb{D}(t)}\int_P\frac{1}{|P^*|}2^{-2j\mu}t^{-n(\frac{2}{r}-1)}\left(\int_{S_j(P)}|\tilde{h}(y)|^rdy\right)^{2/r}\tilde{v}(x)dx\right)^{1/2}\,,
\end{multline}
where in the last step, 
we have used that by \eqref{eq3.27}, we may apply \eqref{3.rcondition}  to $R^{(i)}_t$, since 
$A_{t}$ is a projection, and since $A_t$ and $P_t$ have  compactly supported kernels. 
In turn,  since $|P^*|\approx t^n$, and since $\mu = n/r +\varepsilon$,
the last expression in \eqref{eq3.22} is comparable to
\begin{multline*}
\sum_{j=0}^{\infty}\left(\displaystyle\sum_{P\in\mathbb{D}(t)}\int_P2^{-2j\epsilon}2^{-2jn/r}t^{-2n/r}\tilde{v}(x)\left(\int_{S_j(P)}|\tilde{h}(y)|^rdy\right)^{2/r}dx\right)^{1/2}\\
\approx\displaystyle\sum_{j=0}^{\infty}\left(\displaystyle\sum_{P\in\mathbb{D}(t)}\int_P2^{-2j\epsilon}\tilde{v}(x)\left(\fint_{S_j(P)}|\tilde{h}(y)|^rdy\right)^{2/r}dx\right)^{1/2}\\
\lesssim\, \sum_{j=0}^{\infty}2^{-j\epsilon}
\left(\displaystyle\sum_{P\in\mathbb{D}(t)}\int_P(\mathcal{M}(|\tilde{h}(x)|^r))^{\frac{2}{r}}\tilde{v}(x)dx\right)^{1/2}\\
\lesssim\,\sum_{j=0}^{\infty}2^{-j\epsilon}
\left(\int_{\mathbb{R}^n}|\tilde{h}(x)|^2\tilde{v}(x)dx\right)^{1/2}\,
\approx\,\left(\int_{\mathbb{R}^n}|\tilde{h}(x)|^2\tilde{v}(x)dx\right)^{1/2}\,,
\end{multline*}
where in the next-to-last step, we have used that the maximal operator $\mathcal{M}$ is bounded
on $L^{2/r}_v$, since $v\in A_{2/r}$.
 The proof of Claim \ref{2.claim2} is now complete.
\end{proof}














\subsection{Step 2:  Hypotheses of Lemma \ref{3.lemma} imply hypotheses of Sublemma \ref{3.sublemma}}\label{c4.s2}

\begin{proof}  Recall that
\begin{equation*}
G_Q:= 
\ADQG(\Theta_t 1)\,. 
\end{equation*}
For a large, but fixed $N$ (to be chosen momentarily) let $\Omega_N:=\{x\in Q:G_Q(x)>N\}$.
If the hypotheses of the lemma hold with $E=Q\setminus\displaystyle\bigcup_j Q_j$,
we then have

\begin{align*}
|\Omega_N|&\leq \sum_j|Q_j|+|\{x\in E: G_Q(x)>N\}|\\
&\leq (1-\eta)\left|Q\right|+\frac{1}{N^p}\int_E\left(\ADQG(\Theta_t 1)(x)\right)^{p}dx\\
&\leq \left[(1-\eta)+\frac{C}{N^p}\right]\left|Q\right| \leq (1-\beta) |Q|\,,
\end{align*}

\noindent for some $\beta>0$, where we obtain the last estimate by choosing  
 $N$ large enough, depending on  $\eta$.

\end{proof}

\subsection{Step 3:  Proof of Sublemma \ref{3.sublemma}}\label{c4.s3}

\begin{proof}

Fix (momentarily) a large $M$, so that
$2^{-M}\in(0,1)$, and given a dyadic cube $Q$,  set
$$\DD_M(Q):= \{Q'\in\DD(Q): 2^{-M}\leq \ell(Q')\leq 2^M\}$$
(of course, the upper bound on $\ell(Q')$ is relevant only if $\ell(Q)>2^M$).  
Define a truncated dyadic cone
$$\GD_{Q,M}(x):=\bigcup_{Q'\in \DD_M(Q):\, x\in Q'} U_{Q'}$$
For each dyadic cube $Q$, set
$$G_{Q,M} (x):= \left(\iint_{\GD_{Q,M}(x)}|\Theta_t 1(y)|^2\,
  \mathbbm{1}_{\Gamma^{\epsilon}}\big(\Theta_t 1(y)\big)
\frac{dy dt}{t^{n+1}}\right)^{1/2}\,.$$
Thus, by definition, $G_{Q,M}\leq G_Q:= 
\ADQG(\Theta_t 1)$, and we note that $G_{Q,M}\to G_Q$, as $M\to\infty$, by monotone convergence.
Now let $N, \,\beta$ be as in the hypotheses.
Setting
\begin{equation}\label{eq3.30a}
\Omega_{N,M}=\Omega_{N,M}(Q):= \{x\in Q: G_{Q,M}>N\}\,,
\end{equation}
we have that
\begin{equation}\label{eq3.31}
|\Omega_{N,M}| \leq |\{x\in Q: G_{Q}>N\}|\leq (1-\beta)|Q|\,,
\end{equation}
by hypothesis.
We set

\begin{equation*}K(M):=\sup_Q\frac{1}{|Q|}\int_Q
\left(G_{Q,M}(x)\right)^2dx.\end{equation*}

\noindent  We observe that, for each fixed $M$, $K(M)$ is finite, by the truncation of the cones defining  $G_{Q,M}$.  Our goal is to show that 
\begin{equation}\label{eq3.30}
\sup_{M<\infty}K(M):= K_0<\infty\,.
\end{equation}  
Indeed, in that case, letting $M \to\infty$, we then obtain by
monotone convergence that, for every cube $Q$,
\begin{multline*}K_0 \,|Q|\, \geq\, 
\int_Q\left(G_{Q}(x)\right)^2dx = \int_Q\int\!\!\!\int_{\GD_{Q}(x)}|\Theta_t 1(y)|^2\,
  \mathbbm{1}_{\Gamma^{\epsilon}}\big(\Theta_t 1(y)\big)
\frac{dy dt}{t^{n+1}} \,dx\\[4pt]
 =\, \int_Q\sum_{Q'\in\DD(Q)}1_{Q'}(x)\int\!\!\!\int_{U_{Q'}}
|\Theta_t 1(y)|^2\,
  \mathbbm{1}_{\Gamma^{\epsilon}}\big(\Theta_t 1(y)\big)
\frac{dy dt}{t^{n+1}} \,dx\\[4pt]
= \, \sum_{Q'\in\DD(Q)}\int_{\ell(Q')/2}^{\ell(Q')}\int_{Q'}
|\Theta_t 1(y)|^2\,
  \mathbbm{1}_{\Gamma^{\epsilon}}\big(\Theta_t 1(y)\big)  \left(t^{-n}\int_{Q'} dx\right)
\frac{dy dt}{t} \\[4pt]
\approx\, \int_0^{\ell(Q)}\int_Q|\Theta_t 1(y)|^2\,
  \mathbbm{1}_{\Gamma^{\epsilon}}\big(\Theta_t 1(y)\big) 
\frac{dy dt}{t}\,.
\end{multline*}
Since the latter bound holds for any cone $\Gamma^\epsilon$ of aperture
$\epsilon$, we may sum over a collection of cones $\{\Gamma^\epsilon_k\}$ covering
$\CC^m$, to obtain \eqref{3.conclusionsublemma}.

Thus, it remains only to prove \eqref{eq3.30}.
We fix a cube Q, and let $\Omega_{N,M}$ be the set defined in \eqref{eq3.30a}.
Our first task is to construct a family of stopping time cubes\footnote{
We cannot just use a standard Whitney covering:  since we are working with dyadic cones 
in the definition of $G_{Q,M}$, it is unclear whether
the set $\Omega_{N,M}$ is open.}
covering $\Omega_{N,M}$, modulo a set of measure zero.   We proceed as follows.
Set $F_{N,M}:= Q\setminus \Omega_{N,M},$
and observe that $F_{N,M}$ has positive measure, by \eqref{eq3.31};  thus, in particular, 
the interior of $Q$ (which we denote
int$(Q)$), meets $F_{N,M}$.  We sub-divide
$Q$ dyadically, stopping the first time that we obtain a cube whose interior misses $F_{N,M}$.
In this way, we obtain a family\footnote{Not to be confused with
the family of cubes in the statement of Lemma \ref{l3.7}.  The present cubes have nothing to do with those.} $\{Q_j\}$ of dyadic sub-cubes of $Q$, which are maximal with 
respect to containment of int($Q_j$) 
in $\Omega_{N,M}$.  Obviously, $\cup_j \, {\rm int}(Q_j) \subset \Omega_{N,M}$,
and we wish to establish the opposite containment, up to a set of measure zero.  More precisely,
let $Z$ denote the set of measure zero consisting of the union of the
boundaries of all cubes in $\DD(Q)$.   We claim that
$$\left(\Omega_{N,M}\setminus Z\right)\subset \cup_j \, {\rm int}(Q_j)\,.$$ 
Suppose not:  then
there is a point $x_0\in\Omega_{N,M}$ which does not belong to the interior of
any  $Q_j$, nor to the boundary of any dyadic cube.  
Thus, the interior of every $Q'\in\DD(Q)$, which contains $x_0$, must meet $F_{N,M}$
(otherwise, one such cube would have been a selected $Q_j$).  
Let $Q_0$ be the unique cube in $\DD(Q)$, with side length
$2^{-M-1}$, which contains $x_0$, so that there is a point $x_1$ in $F_{N,M}\cap {\rm int}(Q_0)$.
Notice that for $Q'\in\DD_M(Q)$, we have that $x_0\in Q'$ if and only if $x_1\in Q'$, since in either case,
we must have $Q_0\subset Q'$.  Consequently, $\GD_{Q,M}(x_0) = \GD_{Q,M}(x_1$), and therefore 
$G_{Q,M}(x_0) = G_{Q,M}(x_1)$, which contradicts that $x_0\in \Omega_{N,M}$ and $x_1\in Q\setminus
\Omega_{N,M}$.  We have therefore shown that
$$\Omega_{N,M} = \bigcup_j Q_j\,,$$
up to a set of measure zero, so by \eqref{eq3.31}, we have
\begin{equation}\label{eq3.33a}
\sum_j|Q_j| \leq (1-\beta)|Q|\,.
\end{equation}

With this covering in hand, we proceed to the proof of \eqref{eq3.30}.   
For convenience of notation, we let $\DDT_M(Q)$ denote the collection
of all $Q'\in\DD_M(Q)$ that are not contained in any $Q_j$,
and given $x\in Q$, we set 
$$\DD_M(Q,x) := \{Q'\in\DD_M(Q):\, x\in Q'\}\,,\qquad \DDT_M(Q,x) := \{Q'\in\DDT_M(Q):\, x\in Q'\}\,.$$
Recalling the definition
of the dyadic cones 
$\GD_{Q,M}$, we observe that for $x\in Q_j$, we have
\begin{multline*}
\GD_{Q,M}(x)\,:=\, \bigcup_{Q'\in \DD_M(Q,x)} U_{Q'}
\,=\, \Big(\bigcup_{Q'\in \DD_M(Q_j,x)} U_{Q'}\Big)\,\,\bigcup\,\,
\Big(\bigcup_{Q'\in \DDT_M(Q,x)} U_{Q'}\Big)\\[4pt]
=:\,\GD_{Q_j,M}(x)\,\,\bigcup\,\, \GDT_{Q,M}(x)\,.
\end{multline*}
We then make the corresponding definitions
$$G_{Q_j,M} (x):= \left(\iint_{\GD_{Q_j,M}(x)}|\Theta_t 1(y)|^2\,
  \mathbbm{1}_{\Gamma^{\epsilon}}\big(\Theta_t 1(y)\big)
\frac{dy dt}{t^{n+1}}\right)^{1/2}\,,$$ and
$$\GTQM (x):= \left(\iint_{\GDT_{Q,M}(x)}|\Theta_t 1(y)|^2\,
  \mathbbm{1}_{\Gamma^{\epsilon}}\big(\Theta_t 1(y)\big)
\frac{dy dt}{t^{n+1}}\right)^{1/2}\,.$$
Then by definition, for $x\in Q_j$, we have
$$\left(G_{Q,M}(x)\right)^2 = \left(G_{Q_j,M}(x)\right)^2+\left(\GTQM(x)\right)^2\,,$$
and also, by the stopping time construction,
\begin{equation}\label{eq3.33}
\GTQM(x) = \GTQM(x_1)\leq G_{Q,M}(x_1) \leq N\,,
\end{equation}
for some $x_1\in Q\setminus \Omega_{N,M}$.
Consequently,
\begin{multline*}
\int_Q\left(G_{Q,M}(x)\right)^2 dx\,= \,\int_{F_{N,M}}\left(G_{Q,M}(x)\right)^2 dx\,+
\, \sum_j\int_{Q_j}\left(G_{Q,M}(x)\right)^2 dx\\[4pt]
\leq\,  N^2 |Q| \,+\, 
\sum_j\int_{Q_j}\left(G_{Q_j,M}(x)\right)^2dx\,+\,\sum_j\int_{Q_j}\left(\GTQM(x)\right)^2\\[4pt]
\leq \,N^2 |Q| \,+\, K(M) \sum_j|Q_j| \,+\, N^2\sum_j|Q_j| \\[4pt]
\leq\, 2N^2 |Q| \,+\, K(M) (1-\beta)|Q|\,,
\end{multline*}
where in the last two inequalities we have used the definition of $K(M)$ and \eqref{eq3.33}, 
and then \eqref{eq3.33a}.
Dividing by $|Q|$,  and then taking the supremum over all dyadic $Q$, 
we obtain
$$K(M) \leq 2N^2 +K(M)(1-\beta) \,\implies\, K(M) \leq 2N^2/\beta\,.$$
\end{proof}

\section{Local Tb Theorem without pointwise kernel bounds, version 2}\label{s4}
In this section, we present a version of the local $Tb$ Theorem for square functions,
which may be applied directly to solve the Kato square root problem.
We continue to use the notation introduce in the previous section, but
we shall modify our hypotheses on the operators $\Theta_t$.

\begin{definition}\label{def4.1}

 We consider a family of operators $\{\Theta_t\}_{t>0}$, taking values in $\CC^m,\,m\geq 1$,
 so that  $\Theta_t := (\Theta_t^{1},\Theta_t^{2},
...,\Theta_t^{m})$, and
for $f=(f^1,f^2,...,f^m)\in L^2(\mathbb{R}^n,\CC^m)$, we set
$$ \Theta_tf:=\sum_{j=1}^m \Theta_t^{j}f^j\,.$$ 
We also define the action of $\Theta_t$ on an $m\times m$ matrix valued function $b =(b_{j,k})$,
in the obvious way, i.e., $\Theta_t b=((\Theta_t b)_1,(\Theta_t b)_2,...,(\Theta_t b)_m)$
is a $\CC^m$ valued function, with
$$\left(\Theta_t b\right)_k:= \sum_{j=1}^m \Theta^j_t b_{j,k}\,.$$
We suppose that $\Theta_t$ satisfies the following properties: 

\noindent
\item[(a)] (Uniform $L^2$ bounds and off-diagonal decay in $L^2$).
\begin{equation}
\displaystyle \sup_{t>0}||\Theta_t f||_{L^2(\mathbb{R}^n)}\leq C||f||_{L^2(\mathbb{R}^n)}\,,
\label{4.2}\end{equation}

\begin{equation}
||\Theta_t f_j||_{L^2(Q)}\leq C2^{-j(n+2+\beta)/2}||f_j||_{L^2(2^{j+1}Q\setminus 2^jQ)}\,, \quad 
\ell(Q)\leq t\leq 2\ell(Q)\,,
\label{4.3}\end{equation} 
\noindent for some $\beta>0$, where $f_j:=f\mathbbm{1}_{2^{j+1}Q\setminus 2^jQ}$.

\smallskip

\noindent(b) (Quasi-orthogonality on a subspace of $L^2$).   There exists a subspace $H$ of 
$L^2(\mathbb{R}^n,\CC^m)$, such that for some (hence every) CLP family $\{Q_s\}$ (cf. 
Definition \ref{3.lpfamily}), 
there is a $\beta>0$ for which we have
\begin{equation}
||\Theta_t Q_sh||_{L^2(\mathbb{R}^n)}\leq C\left(\frac{s}{t}\right)^{\beta}||h||_{L^2(\mathbb{R}^n)}, \quad 
\forall h\in H\,,\,\, \forall s\leq t\,.
\label{4.4}\end{equation}

\end{definition}

\begin{remark}\label{remark 4.1} Let us note that it is only in the quasi-orthogonality
condition (b), that we 
restrict to a sub-space $H\subset L^2(\mathbb{R}^n,\CC^m)$;  in condition (a),
we suppose that the operator is allowed to act on general $C^m$-valued functions
in $L^2$.
\end{remark}

\begin{definition}\label{def4.2}  Given a subspace $H\subset L^2(\mathbb{R}^n,\CC^m)$,
we let $H^m$ denote the subspace of $L^2(\RR^n, \mathbb{M}^m)$ 
(i.e., complex $m\times m$ matrix-valued $L^2$),
consisting of $L^2$ matrix-valued functions $b =(b_{j,k})$ 
for which each column vector $b^k:=(b_{1,k}, b_{2,k},..., b_{m,k})$ belongs to $H$.
\end{definition}

Our main result in this section is the following:  
\begin{theorem}\label{t4.2}
Let $\{\Theta_t\}_{t>0}$, and the subspace $H\subset L^2(\RR^n,\CC^m)$, 
be as in Definition \ref{def4.1} above,
 and suppose that there exist  positive constants $C_0<\infty$, and $\delta>0$, 
and a system $\{b_Q\}\subset H^m$, 
indexed by  dyadic
cubes $Q\subset\mathbb{R}^n$,
such that for each dyadic cube Q:
\begin{equation}
\int_{\mathbb{R}^n}|b_Q(x)|^2dx\leq C_0|Q|,
\label{4.5}\end{equation}
\begin{equation}
\int_0^{\ell(Q)}\!\!\int_Q |\Theta_t b_Q(x)|^2 \frac{dx dt}{t}\leq C_0|Q|,
\label{4.6}\end{equation}
\begin{equation}
\delta\,|\xi|^2 
\,\leq\, Re\left(\xi\cdot\fint_Qb_Q(x) dx 
\, \bar{\xi}\right), \qquad \forall\xi\in\mathbb{C}^m \,.
\label{4.7}\end{equation}
Then
\begin{equation}
\iint_{\mathbb{R}_+^{n+1}}|\Theta_tf(x)|^2\frac{dxdt}{t}\leq C||f||_{L^2(\mathbb{R}^n)}^2,\qquad
\forall f\in H.
\label{4.8}\end{equation}
\end{theorem}

The proof of Theorem \ref{t4.2} will be based upon
 the following generalization of the Christ-Journ\'e $T1$ theorem for square functions.

\begin{theorem}(T1 Theorem)
Suppose that $\Theta_t$ satisfies conditions \eqref{4.2}, \eqref{4.3} 
and \eqref{4.4} as above, 
and also that we have the Carleson measure estimate
\begin{equation}
\displaystyle\sup_Q\frac{1}{|Q|}\int_Q\int_0^{\ell(Q)}|\Theta_t1(x)|^2\frac{dxdt}{t}\leq C.
\label{4.10}\end{equation}
Then we have the square function estimate
\begin{equation}
\iint_{\mathbb{R}_+^{n+1}}|\Theta_tf(x)|^2\frac{dxdt}{t}\leq C||f||_{L^2(\mathbb{R}^n)}^2\,,\quad \forall f\in H.
\end{equation}
\label{4.T1theorem}\end{theorem}

As above, we view ``1" in this context as the $m\times m$ identity matrix.
We defer the proof of Theorem \ref{4.T1theorem} to Section \ref{st1}.

\begin{proof}[Proof of Theorem \ref{t4.2}]
By Theorem \ref{4.T1theorem}, it suffices to establish \eqref{4.10}.
As in the previous section, we choose an $\epsilon>0$ small enough that Lemma
\ref{l2.6} holds, and  we cover $\CC^m$ by a family of 
cones $\{\Gamma^\epsilon_k\}_{1\leq k\leq K(\epsilon,M)}$,
of  aperture $\epsilon$.   It suffices to
establish the uniform estimate
\begin{equation}\label{4.12}\sup_Q\frac1{|Q|}\int_0^{\ell(Q)}\!\!\int_Q
|\Theta_t1(x)|^2 \mathbbm{1}_{\Gamma^\epsilon}\left(\Theta_t 1(x)\right)\frac{dxdt}{t}\,
\leq \,C\,,\end{equation}
for any fixed cone $\Gamma^\epsilon$, since in that case we may apply the latter estimate to each of
the cones $\Gamma_k^\epsilon$, and then sum in $k$ to obtain \eqref{4.10}.

Let us then fix a cube $Q$ and a cone $\Gamma^\epsilon$.  We invoke Lemma 2.6
to obtain a constant $\eta>0$, and a family of non-overlapping cubes $\{Q_j\}\subset \DD(Q)$  satisfying
\begin{equation}\label{4.13}
\sum_j|Q_j| \leq (1-\eta)|Q|\,,
\end{equation}
and
\begin{equation}\label{4.14}
|\Theta_t1(x)|^2 \mathbbm{1}_{\Gamma^{\epsilon}}\big(\Theta_t1(x)\big)
\leq 4|\Theta_t1(x) A_{t}b_Q(x)|^2\,,\qquad \forall (x,t) \in E_Q^*:= R_Q\setminus
\Big(\bigcup_j R_{Q_j}\Big)\,,
\end{equation} 
where we recall that for any cube $Q$, $R_Q:=Q\times (0,\ell(Q))$ is the corresponding Carleson box.
Now, by a well-known John-Nirenberg type lemma for Carleson 
measures (see, e.g., \cite[Lemma 1.37]{H1}),  the left hand side of  \eqref{4.12} is bounded
by a constant times
$$\sup_Q\frac1{|Q|}\int\!\!\!\int_{E^*_Q}
|\Theta_t1(x)|^2 \mathbbm{1}_{\Gamma^\epsilon}\left(\Theta_t 1(x)\right)\frac{dxdt}{t} \,.$$
In turn, by \eqref{4.14}, we have
$$\int\!\!\!\int_{E^*_Q}
|\Theta_t1(x)|^2 \mathbbm{1}_{\Gamma^\epsilon}\left(\Theta_t 1(x)\right)\frac{dxdt}{t} \,
\leq\, 4\int\!\!\!\int_{R_Q}
|\Theta_t1(x) A_{t}b_Q(x)|^2 \frac{dxdt}{t} \,,$$
so it suffices to show that the last expression is no larger than a uniform constant time $|Q|$.
To this end, as usual we write
\begin{multline*}\Theta_t1 A_{t}b_Q = \big((\Theta_t1) A_{t}-(\Theta_t1)A_tP_t\big) b_Q\,+
 \big((\Theta_t1) A_{t}P_t-\Theta_t\big) b_Q\,
+\, \Theta_t b_Q\\[4pt]=: R^{(1)}_t b_Q \,+\,R_t^{(2)}b_Q\,+\,\Theta_tb_Q\,,
\end{multline*}
where as above $P_t$ denotes a nice approximate identity operator
(cf. Definition \ref{Ptdefinition}).  By \eqref{4.6}, the contribution of $\Theta_t b_Q$ immediately
satisfies the desired bound.  We may handle the contributions of the two ``remainder" terms
$R_t^{(1)}b_Q$ and $R_t^{(2)}b_Q$, by invoking \eqref{4.5} and Corollary
\ref{c3.11}, which in the present setting is applied only to $h\in H$ 
(recall that, in particular, each column vector
of $b_Q$ belongs to $H$, by hypothesis).    
\end{proof}

\section{Application to the Kato square root problem}\label{c6}

Let A be an $n\times n$ matrix of complex, $L^{\infty}$ coefficients, defined on $\mathbb{R}^n$, and satisfying the ellipticity (or ``accretivity") condition
\begin{equation*}
\lambda|\xi|^2\leq\mathcal{R}e<A\xi,\xi>\equiv\displaystyle\sum_{i,j}A_{ij}(x)\xi_j\bar{\xi}_i, \ \Vert A\Vert_{\infty}\leq\Lambda,
\end{equation*}

\noindent for $\xi\in\mathbb{C}^n$ and for some $\lambda,\Lambda$ such that $0<\lambda\leq\Lambda<\infty$. We define the divergence form operator
\begin{equation*}
Lu:= -\dv\left(A(x)\nabla u\right),
\end{equation*}
\noindent which we interpret in the usual weak sense via a sesquilinear form.

The accretivity condition above enables one to define an accretive square root $\sqrt{L}$.

\begin{theorem}\cite{AHLMcT} \label{th6.1}
Let L be a divergence form operator defined as above.
Then for all $h\in\dot{L}^2_1(\mathbb{R}^n)$, we have 
\begin{equation*}
\Vert\sqrt{L}h\Vert_{L^2(\mathbb{R}^n)}\leq C\Vert\nabla h\Vert_{L^2(\mathbb{R}^n)}, 
\end{equation*}
with C depending only on $n,
\lambda$ and $\Lambda$.
\end{theorem}

\begin{proof}


In \cite{AT}, it is shown that the conclusion of Theorem \ref{th6.1} is equivalent to 
the square function estimate
$$\iint_{\mathbb{R}^{n+1}_+}|\Theta_t\nabla h|^2 \frac{dxdt}{t} \leq C_{n ,\lambda,\Lambda}\,\int_{\mathbb{R}^n}|\nabla h|^2 dx\,,$$
where  $\Theta_t := te^{-t^2L}\dv A$.  Thus,
to prove this theorem, it is enough to verify the conditions of
the Local Tb Theorem (Theorem \ref{t4.2}), 
for this operator $\Theta_t$, with $m=n$, and with 
$$H:=\{\nabla h: \ h\in\dot{L}_1^2(\mathbb{R}^n,\mathbb{C}^n)\}\,,$$ 
a subspace of $L^2(\mathbb{R}^n,\mathbb{C}^n).$

\noindent (a) It follows from standard semigroup theory for divergence form
elliptic operators that the 
family  $\{t e^{-t^2L}\dv\}_{t>0}$ satisfies uniform $L^2$ bounds,
as well as $L^2$ off-diagonal estimates 
(for the latter, just dualize the well known ``Gaffney estimates" for $t\nabla e^{-t^2L}$).
Thus, since $A\in L^{\infty}(\mathbb{R}^n)$, we have condition (a) for $\Theta_t$.

\noindent (b) We choose $\{Q_s\}_{s>0}$ of convolution type satisfying the required conditions
of Definition \ref{3.lpfamily}, and in addition we choose the kernel $\psi \in C_0^\infty$.
We note that
$Q_s\nabla F=\nabla Q_sF$.
Let $\nabla F\in H$, where (by definition) $F\in\dot{L}_1^2(\mathbb{R}^n)$.
Then $F$ is equivalent (modulo constants) to the realization 
$F=I_1f$, where $f\in L^2(\mathbb{R}^n)$,
with $\|f\|_2\approx\|\nabla F\|_2$,
and $I_1$ is the Riesz potential with $\alpha=1$ as in Definition \ref{Rieszpotential}.
We then have
\begin{align*}
te^{-t^2L}\dv AQ_s\nabla F&=te^{-t^2L}\dv A\nabla Q_sF\\
&=-tLe^{-t^2L}Q_sF\\
&=-\frac{1}{t}t^2Le^{-t^2L}Q_sI_1f\\
&=\frac{-s}{t}t^2Le^{-t^2L}\left(\frac{1}{s}Q_sI_1\right)f.
\end{align*}

Using the fact that $t^2Le^{-t^2L}:L^2\rightarrow L^2$ and $\frac{1}{s}Q_sI_1:L^2\rightarrow L^2$ we obtain condition (b).

Finally, we need to find a family of $b_Q$ indexed by cubes Q satisfying the required conditions. In \cite{HMc}, \cite{HLMc} and \cite{AHLMcT}, it is proved that such a family exists,
with $b_Q$ of the form 
$$b_Q=\nabla F_Q:= e^{-\varepsilon \ell(Q)^2 L}(\varphi_Q)\,,$$
with $ \varphi_Q (x):= (x-x_Q)\eta_Q$, where $\varepsilon >0$, is a number chosen suitably small depending
only on   $n,\lambda$ and $\Lambda$, and where $x_Q$ denotes the center of $Q$, and $\eta_Q$
is a smooth cut-off function which is 1 on $4Q$ and vanishes outside of $5Q$.  We refer the reader to 
\cite{AHLMcT} for details of the proofs of the required properties.

\end{proof}

\section{A local $Tb$ theorem with vector-valued testing functions}\label{s6}

In this section, we consider a version of the local $Tb$ theorem for vector-valued
$\Theta_t$, in which the testing functions themselves are vector-valued rather than matrix-valued.
In the following section, we shall then apply this version to establish $L^2$ bounds for layer potentials
associated to a class of divergence form elliptic operators in the half-space $\uh$.

\begin{definition}\label{def6.1}
 We consider a family of operators $\{\Theta_t\}_{t>0}$, taking values in $\CC^{m+1},\,m\geq 1$,
 so that  $\Theta_t := (\Theta_t^0,\Theta'_t):=(\Theta_t^0,\Theta_t^{1},\Theta_t^{2},
...,\Theta_t^{m})$, where each $\Theta_t^j,\, 0\leq j\leq m$, acts on scalar valued $L^2$, and where
for $g=(g^0,g'):=(g^0,g^1,,...,g^m)\in L^2(\mathbb{R}^n,\CC^{m+1})$, we set
$$ \Theta_t g=\sum_{j=0}^m \Theta_t^{j}g^j\,,\qquad \Theta'_tg'=\sum_{j=1}^m \Theta_t^{j}g^j.$$ 
We suppose that $\Theta_t$ satisfies the following properties:

\noindent
\item[(a)] (Uniform $L^2$ bounds and off-diagonal decay in $L^2$).
\begin{equation}
\displaystyle \sup_{t>0}||\Theta_t g||_{L^2(\mathbb{R}^n)}\leq C||g||_{L^2(\mathbb{R}^n,\CC^m)}\,,
\label{6.5a}\end{equation}
\begin{equation}
||\Theta_t g_j||_{L^2(Q)}\leq C2^{-j(n+2+\beta)/2}||g_j||_{L^2(2^{j+1}Q\setminus 2^jQ)}\,, \quad 
\ell(Q)\leq t\leq 2\ell(Q)\,,
\label{6.6a}\end{equation} 
\noindent for some $\beta>0$, where $g_j:=g\mathbbm{1}_{2^{j+1}Q\setminus 2^jQ}$;

\smallskip

\noindent(b) (Quasi-orthogonality on a subspace of $L^2$).   There exists a subspace $H$ of 
$L^2(\mathbb{R}^n,\CC^m)$, such that for some (hence every) CLP family $\{Q_s\}$ (cf. 
Definition \ref{3.lpfamily}), 
there is a $\beta>0$ for which we have
\begin{equation}
||\Theta_t Q_sh||_{L^2(\mathbb{R}^n)}\leq C\left(\frac{s}{t}\right)^{\beta}||h||_{L^2(\mathbb{R}^n)}, \quad 
\forall h=(h^0,h')\in L^2(\rn,\CC)\times H\,,\,\, \forall s\leq t\,.
\label{6.8}\end{equation}


We further define the action of $\Theta_t$ (respectively, $\Theta_t'$)  on matrix-valued
$B=(B_{ij})_{0\leq i,j\leq m}:\RR^n\to\MM^{m+1}$ (resp., 
$B'=(B_{ij})_{1\leq i,j\leq m}:\RR^n\to\MM^m$) in the obvious way:
\begin{equation}
\Theta_t B(x)=\left(\sum_{i=0}^m \Theta_t^i(x,y)\, B_{ij}\right)_{\! 0\leq j\leq m}\,\quad
\Theta'_t B'(x)=\left(\sum_{i=1}^m \Theta_t^i(x,y)\, B_{ij}\right)_{\! 1\leq j\leq m}\,.
\label{6.7}\end{equation} 
\end{definition}

For $\Theta_t$ as above, for convenience of notation, we set
$$\zeta(x,t):=\Theta_t 1(x)\,,\quad \zeta'(x,t):=\Theta_t' 1(x)\,,\quad \zeta^0(x,t):=\Theta_t^0 1(x)\,,$$
where in the previous line, ``1" denotes, respectively, the $(m+1)\times(m+1)$ identity matrix.
the $m\times m$ identity matrix, and the scalar constant 1.  For $\delta \in [0,1)$, we also set
$$\zeta_\delta(x,t):=\zeta(x,t)1_{\{\delta<t<1/\delta\}}\,,\quad 
\zeta'_\delta(x,t):=\zeta'(x,t)1_{\{\delta<t<1/\delta\}}\,,\quad \zeta_\delta^0(x,t):=\
\zeta^0(x,t)1_{\{\delta<t<1/\delta\}}\,.$$

\begin{definition}\label{def6.8}  Given $\Upsilon$ (either scalar or vector valued),
defined on $\uh$, we set
$$\|\Upsilon\|_{\C}:= \sup_Q\frac1{|Q|} \int_0^{\ell(Q)}\!\!\int_Q|\Upsilon(x,t)|^2\frac{dxdt}{t}\,.$$
\end{definition}

\begin{definition}\label{def6.12}  Given a dyadic cube $Q$, and a positive constant $C_0$,
we shall say that a Borel measure 
$\mu_Q$ on $\rn$
is ``adapted to Q with constant $C_0$", if  
$$d\mu_Q \,=\, \phi_Q \, dx\,,$$
where $\phi_Q:\rn\to [0,1]$ is a Lipschitz function satisfying
\begin{equation}
\Vert \nabla\phi_Q\Vert_{L^\infty(\rn)}\leq C_0 \, \ell(Q)^{-1}\,, 
\label{6.13a}\end{equation}
and also
\begin{equation}
\frac1{C_0}\leq \phi_Q\,, \quad {\rm  on} \,\, Q
\label{6.14a}\end{equation}
\end{definition}
We observe that $\phi_Q$, being Lipschitz,  is not  supported in $Q$, by \eqref{6.14a}.

\begin{theorem}\label{t6.1}

Let  $\Theta_t$ and $H$ be as in Definition \ref{def6.1}, and suppose that there exist positive constants 
$\sigma$, 
$\eta$,  $C_0$ and $C_1$,  a system $\{b_Q\}=\{(b_Q^0,b_Q')\} \subset 
L^2(\rn,\CC) \times H \subset
L^2(\rn,\CC^{m+1})$,  indexed by dyadic cubes $Q\subset\mathbb{R}^n$, 
and a 
system of measures $\{\mu_Q\}$, also indexed by  dyadic 
cubes, with each $\mu_Q$ ``adapted to $Q$ with constant $C_0$" as in Definition
\ref{def6.12},
such that for every dyadic cube Q, we have
\begin{equation}
\int_{\mathbb{R}^n}|b_Q(x)|^2dx\leq C_0|Q|,
\label{6.9}\end{equation}
\begin{equation}
\int_0^{\ell(Q)}\!\!\int_Q|\Theta_tb_Q(x)|^2\frac{dxdt}{t}\leq C_0|Q|,
\label{6.10}\end{equation}
\begin{equation}
\sigma\leq \mathcal{R}e\left(\fint_Qb^0_Q\,d\mu_Q\right), 
\label{6.11}\end{equation}
\begin{equation}
 \left|\fint_Qb'_Q\,d\mu_Q\right|\leq \eta\sigma, 
\label{6.12}\end{equation}
with $\eta\leq1/(2C_1+4)$, where,
for all sufficiently small $\delta >0$ (hence also for $\delta=0$),
\begin{equation}
\|\zeta_\delta'\|_{\C}\leq C_1\big(1+\|\zeta_\delta^0\|_{\C}\big)\,. 
\label{6.13}\end{equation}
Then for all $f^0\in L^2(\rn,\CC)$,
\begin{equation}
\iint_{\mathbb{R}_+^{n+1}}|\Theta^0_t f^0(x)|^2\frac{dxdt}{t}\leq C||f^0||_2^2.
\label{6.14}\end{equation}
\end{theorem}

\begin{remark}  Given the hypotheses of Theorem \ref{t6.1},
we also have that for all $h'\in H$,
\begin{equation}\label{eq6.16aa}
\iint_{\mathbb{R}_+^{n+1}}|\Theta'_t h'(x)|^2\frac{dxdt}{t}\leq C||h'||_2^2.
\end{equation}
Indeed, by a well known argument of \cite{FS}, \eqref{6.14} 
implies that $\|\Theta^0_t 1\|_{\mathcal{C}}\leq C$, whence \eqref{eq6.16aa}
follows immediately by \eqref{6.13} and Theorem \ref{4.T1theorem}.
\end{remark}

\begin{remark}  \label{remark6.2} 
It may be that in some applications, one could 
simply take $\phi_Q\equiv 1$, i.e., $d\mu_Q= dx$,
but in our application in Section \ref{s7}, it is useful to have the extra flexibility inherent in 
\eqref{6.11}-\eqref{6.12}. \end{remark}

\begin{remark} \label{remark6.3}
As a practical matter, in applications, one expects to construct
$b_Q'$ depending on $\eta$, where the parameter $\eta$ in \eqref{6.12} is at one's disposal and can 
therefore be made sufficiently small.   This is precisely what we shall do in applying Theorem
\ref{t6.1} to the study of variable coefficient layer potentials, in Section \ref{s7}.  We emphasize that
$\eta$ is required to be small depending only upon $C_1$, but not on $C_0$;   thus, in applications,
there is no harm if $C_0$ becomes larger as $\eta$ becomes smaller, as indeed, will be the case in
our application in Section \ref{s7}.
\end{remark}

\begin{proof} By Theorem \ref{3.T1theorem} (the generalized $T1$ Theorem of \cite{CJ}), it 
is enough to show that $\|\zeta^0\|_{\C}\leq C$.  
We proceed under the {\it a priori} qualitative assumption that $\|\zeta^0\|_\C$ is finite,
but with no particular quantitative bound.  We may do this by working with $\zeta_\delta$
in place of $\zeta$, as long as we obtain bounds that are independent of $\delta$.  To simplify the
notation, we shall simply write $\zeta$, not $\zeta_\delta$.
We define
\begin{align*}
F_1&:= \{(x,t)\in\uh: |\zeta^0(x,t)|\,\leq\sqrt{\eta}\,|\zeta'(x,t)|\}\\[4pt]
F_2&:= \{(x,t)\in\uh: |\zeta^0(x,t)|\,>\sqrt{\eta}\,|\zeta'(x,t)|\}\,,
\end{align*}
so that 
\begin{equation}\label{6.15}
\|\zeta^0\|_\C \,\leq\, \|\zeta^01_{F_1}\|_\C\,+\,\|\zeta^01_{F_2}\|_\C\,.
\end{equation}
By Definition \ref{def6.8}, and \eqref{6.13}, we have
$$\|\zeta^01_{F_1}\|_\C \,\leq \,\eta\,\|\zeta'\|_\C \,\leq\,C_1\eta\big( 1+\|\zeta^0\|_C\big)\,.$$
Since $\eta \leq 1/(2C_1)$, we may hide the term $C_1\eta\,\|\zeta^0\|_\C$ on the left hand side of
\eqref{6.15}.  Thus, it suffices to show that $\|\zeta^01_{F_2}\|_\C\leq C$.  To this end, we first
 note that by \eqref{6.11} and \eqref{6.12},
$$\sigma |\zeta^0|\leq \big|\zeta^0 \fint_Qb^0_Q\,d\mu_Q\big|
\leq \big|\zeta\cdot\fint_Qb_Q\,d\mu_Q\big| +\big|\zeta'\cdot\fint_Qb'_Q\,d\mu_Q\big|
\leq \big|\zeta\cdot\fint_Qb_Q\,d\mu_Q\big| +\eta\sigma|\zeta'|\,,$$
for every dyadic cube $Q$.  In $ F_2$, we then have (again for every $Q$)
$$\sigma |\zeta^0|\leq \big|\zeta\cdot\fint_Qb_Q\,d\mu_Q\big| +\sqrt{\eta}\sigma|\zeta^0|\,,$$
and also
$$|\zeta|\leq |\zeta^0|+|\zeta'|\leq (1+\eta^{-1/2})|\zeta^0|\leq 2\eta^{-1/2}|\zeta^0|\,.$$
Combining the latter two estimates, we obtain that for all $ (x,t)\in F_2$, and for every $Q$,
\begin{equation}\label{6.16}
\frac{\sqrt{\eta}}{2}\big(1-\sqrt{\eta}\big)\sigma |\zeta(x,t)|\leq \big(1-\sqrt{\eta}\big)\sigma|\zeta^0(x,t)|\leq
\big|\zeta(x,t)\cdot\fint_Qb_Q\,d\mu_Q\big|\,. 
\end{equation}
We now observe that, as above, in order to estimate $\|\zeta^01_{F_2}\|_\C$,
it suffices to prove that, for $\epsilon >0$ chosen small enough,
\begin{equation}\label{6.17}
\|\zeta^01_{F_2}\mathbbm{1}_{\Gamma^\epsilon}(\zeta)\|_\C\leq C\,,
\end{equation}
 where $\Gamma^\epsilon$
 is an arbitrary cone of aperture $\epsilon$, i.e.,
 $$\Gamma^\epsilon  =\Gamma^\epsilon(\nu):= \{z\in\CC^{m+1}:|(z/|z|)-\nu|<\epsilon\}\,,$$
for some unit direction vector $\nu\in \CC^{m+1}$.  Indeed, given \eqref{6.17}, we may then sum over an
 appropriate collection of such cones, covering $\CC^{m+1}$, to obtain the bound
 $\|\zeta^01_{F_2}\|_\C\leq C(\epsilon, m).$   We therefore fix such a cone $\Gamma^\epsilon
 =\Gamma^\epsilon(\nu)$,
 and observe that, for $(x,t)\in F_2$,  and 
 $\zeta(x,t)\in\Gamma^\epsilon$, since $\eta\leq 1/4$, by \eqref{6.16}, we have
for every dyadic cube $Q$,
 \begin{multline*}\frac{\sqrt{\eta}}{4}\sigma  \leq \Big|\frac{\zeta(x,t)}{|\zeta(x,t)|}\cdot\fint_Qb_Q\,
 d\mu_Q\Big|
 \leq \Big|\left(\frac{\zeta(x,t)}{|\zeta(x,t)|}-\nu\right)\cdot\fint_Qb_Q\,d\mu_Q\Big|
 +\big|\nu\cdot\fint_Qb_Q\,d\mu_Q\big|\\[4pt]
 \leq\, C_0\epsilon \,+\,\big|\nu\cdot\fint_Qb_Q\,d\mu_Q\big|\,,
 \end{multline*}
where in the last step we have used Schwarz's inequality, the fact 
that $1/C_0\leq d\mu_Q/dx  =\phi_Q\leq 1$ on $Q$,  and \eqref{6.9}.  
Since $\epsilon$ is at 
our disposal, we may choose it small enough, say $\epsilon \leq \sigma\sqrt{\eta}/(8C_0)$, 
and then hide the small term, to obtain
\begin{equation}\label{6.18}
\frac{\sqrt{\eta}}{8}\sigma\,=:\,\theta\,\leq\,\big|\nu\cdot\fint_Qb_Q\,d\mu_Q\big|\,.
\end{equation}
We observe that the latter bound does not depend on $(x,t)$, but on 
the other hand, was deduced from the existence of some $(x,t)\in F_2$, for which
$\zeta(x,t) \in \Gamma^\epsilon$.  Thus, \eqref{6.18} holds for any $Q$ such that
\begin{equation*}
\iint_{R_Q} |\zeta^0(x,t)|^2 1_{F_2}(x,t)\mathbbm{1}_{\Gamma^\epsilon}\big(\zeta(x,t)\big) 
\frac{dx dt}{t}\,\neq \,0\,,
\end{equation*} 
where as usual, $R_Q:=Q\times(0,\ell(Q))$ 
denotes the standard Carleson box above $Q$.
Consequently, in proving \eqref{6.17}, we may suppose henceforth,
without loss of generality, that \eqref{6.18} holds in any dyadic cube $Q$ under consideration.

We therefore fix a dyadic cube $Q$ such that \eqref{6.18} holds, and
we  follow a (now familiar) stopping time procedure to extract
a family $\{Q_j\}$ of non-overlapping dyadic sub-cubes of $Q$,
which are maximal with respect to the property that at least one of the following conditions
holds:
\begin{equation}\label{6.19}
\fint_{Q_j}|b_Q|\,d\mu_Q\,>\,\frac{\theta}{4\epsilon} \,\qquad  {\rm (type \ I) }
\end{equation}
\begin{equation}\label{6.20}
\big|\nu\cdot\fint_{Q_j}b_Q\,d\mu_Q\big|\,\leq\,\frac{\theta}{2} \qquad {\rm(type \ II)}\,.
\end{equation}
If some $Q_j$ happens to satisfy both the type I and type II conditions, then we arbitrarily assign it to
be of type II;  for simplicity of notation, we write $Q_j\in I$, or $Q_j\in II$, to mean that the cube is of
type I, or type II, respectively.
We now claim that for this family of cubes,  
\begin{equation}\label{6.21}
\sum_j|Q_j|\leq(1-\beta)|Q|\,,
\end{equation}
for some uniform $\beta>0$, and that
\begin{equation}\label{6.22}
|\zeta(x,t)|^2 \mathbbm{1}_{\Gamma^{\epsilon}}\big(\zeta(x,t)\big)
\leq \,C_\theta\,|\zeta(x,t)\cdot A^{\mu_Q}_{t}b_Q(x)|^2\,,\qquad \forall (x,t) \in  E_Q^*:=R_Q\setminus
\big(\bigcup_j R_{Q_j}\big)\,,
\end{equation}
where $ A^{\mu_Q}_t$ is the dyadic averaging operator adapted to the measure $\mu_Q$, i.e.,
$$A^{\mu_Q}_t f(x) :=\fint_{Q(x,t)} f\,d\mu_Q\,,$$
where  $Q(x,t)$ denotes the smallest dyadic cube, of side length at least $t$, that contains $x$.

Let us verify \eqref{6.22} first.  Observe first that for any dyadic sub-cube $Q'\subset Q$, which is not contained in any $Q_j$, we have that the opposite inequalities to \eqref{6.19} and \eqref{6.20}
hold, for the average of $b_Q$ over $Q'$, by maximality of the cubes in the family $\{Q_j\}$;  i.e., by definition of $A^{\mu_Q}_t$, we have 
\begin{equation*}
\frac{\theta}{2}\,\leq \big|\nu\cdot A^{\mu_Q}_t(b_Q)(x)\big|\,\, \,{\rm and }\,\, 
\,|A^{\mu_Q}_t(b_Q)(x)|\leq 
\frac{\theta}{4\epsilon}
\,,\qquad \forall (x,t) \in E_Q^*\,.
\end{equation*}
Thus, if $z\in\Gamma^\epsilon$, and $(x,t)\in E_Q^*$, we have
\begin{multline*}
\frac{\theta}{2}\,\leq \big|\nu\cdot A^{\mu_Q}_t(b_Q)(x)\big|\,\leq\,
 \big|(z/|z|)\cdot A^{\mu_Q}_t(b_Q)(x)\big| + \big|\big(\nu-(z/|z|)\big)\cdot A^{\mu_Q}_t(b_Q)(x)\big|\\[4pt]
 \leq\,\big|(z/|z|)\cdot A^{\mu_Q}_t(b_Q)(x)\big| +  \frac{\theta}{4}\,.
 \end{multline*}
We now obtain \eqref{6.22}, by setting  $z=\zeta(x,t)$.

Next, we establish \eqref{6.21}.   Set $E:= Q\setminus (\cup_j Q_j)$, and
 $ B_1:= \cup_{Q_j \in I}Q_j$.  
 We note that by definition of the type I cubes, 
 $ B_1\subset \{M(b_Q) > \theta/(4\epsilon)\}$, where $M$ denotes the Hardy-Littlewood
 maximal operator, whence it follows by \eqref{6.9} that, 
\begin{equation} \label{6.23}
| B_1| \,\leq\, C \left(\frac{\epsilon}{\theta}\right)^2 
\int_{\rn} |b_Q|^2 \,\leq \, 
CC_0\left(\frac{\epsilon}{\theta}\right)^2|Q|\,.
\end{equation}
Let us note that by Definition \ref{def6.12},
\begin{equation}\label{6.30}
\frac1{C_0}\,|Q|\,\leq\,\mu_Q(Q)\,\leq\, |Q|\,.
\end{equation}
 By \eqref{6.18}, and then Definition \ref{def6.12},
 \eqref{6.20}, \eqref{6.23}, and \eqref{6.9}, we have
 \begin{multline*}
 \theta\, \mu_Q(Q)\,\leq\, \big|\nu\cdot\int_{Q}b_Q\,d\mu_Q\big|\\[4pt]\leq \big|\nu\cdot\int_{E}b_Q
 \,d\mu_Q\big|
 \,+\int_{ B_1}|b_Q|\,d\mu_Q\,+\sum_{Q_j \in II}\big|\nu\cdot\int_{Q_j}b_Q\,d\mu_Q\big|\\[4pt]
 \leq\, |E|^{1/2}\, \|b_Q\|_2\, +\,| B_1|^{1/2} \|b_Q\|_2 \,+\, \frac{\theta}{2}\sum_j\mu_Q(Q_j)\,\leq\,
C\, |E|^{1/2}\, |Q|^{1/2}\, +\,C_\theta\epsilon\,|Q| \,+\, \frac{\theta}{2}\mu(Q)\,.
  \end{multline*}
Choosing $\epsilon$ small enough, and using \eqref{6.30},
we have $|Q|\,\leq\, C_\theta \,|E|,$
  which is equivalent to \eqref{6.21}.
  
  With \eqref{6.21} and \eqref{6.22} in hand, we turn to the proof of \eqref{6.17}.
  We  note that by \eqref{6.21}, and a standard John-Nirenberg lemma
  for Carleson measures (see, e.g., \cite[Lemma 1.37]{H1}), 
  \begin{multline*}
  \|\zeta^01_{F_2}\mathbbm{1}_{\Gamma^\epsilon}(\zeta)\|_\C\,\lesssim\,
  \sup_Q \frac1{|Q|}\iint_{E^*_Q}|\zeta^0(x,t)|^2 1_{F_2}(x,t) \mathbbm{1}_{\Gamma^\epsilon}(\zeta(x,t))
  \frac{dx dt}{t}\\[4pt]
  \lesssim \, \sup_Q \frac1{|Q|}\iint_{R_Q}|\zeta(x,t)\cdot A^{\mu_Q}_{t}b_Q(x)|^2
  \frac{dx dt}{t}\,,
  \end{multline*} 
where in the last step we have used the trivial bound $|\zeta^0|\leq  |\zeta|$, and \eqref{6.22}.

Thus it remains only to show that the last term is bounded.
To this end,  recalling that $\zeta=\Theta_t 1$,  we fix $Q$, and use the familiar 
trick of Coifman-Meyer to write
\begin{multline*} \zeta \cdot A^{\mu_Q}_t b_Q = \big((\Theta_t1) 
A^{\mu_Q}_{t}-(\Theta_t1)A^{\mu_Q}_tP_t\big) b_Q\,+
 \big((\Theta_t1) A^{\mu_Q}_{t}P_t-\Theta_t\big) b_Q\,
+\, \Theta_t b_Q\\[4pt]=: R^{(1)}_t b_Q \,+\,R_t^{(2)}b_Q\,+\,\Theta_tb_Q\,,
\end{multline*}
where as above $P_t$ denotes a nice approximate identity operator
(cf. Definition \ref{Ptdefinition}).  The contribution of $\Theta_t b_Q$ may be handled immediately
by hypothesis \eqref{6.10}.  The contributions of the two remainder terms $R_t^{(i)}b_Q,\, i=1,2$,
may be handled as follows.  By \eqref{6.9}, and a standard orthogonality argument,
it is enough to show that for some $\beta_0>0$, and for all $t\in (0, \ell(Q))$, we have
\begin{equation}\label{6.31}
\int_{Q}|R^{(i)}_tQ_sh(y)|^2dy 
\lesssim \min\left(\frac{s}{t},\frac{t}{s}\right)^{\beta_0}\int_{\mathbb{R}^n}|h(x)|^2dx\,.
\end{equation}
In turn, one may prove the latter fact by following {\it mutatis mutandi} the proof of Lemma
\ref{l3.10}, bearing in mind that in $R_Q$, by Definition \ref{def6.12},
the modified dyadic averages $A_t^{\mu_Q}$
enjoy the same estimates as do the standard dyadic averages $A_t$.
\end{proof}

\section{Application of Theorem \ref{t6.1} to the theory of layer potentials}\label{s7}

We  consider layer potentials associated to  divergence form complex coefficient equations
$Lu=0$, where \begin{equation*} L=-\dv A\nabla:=-\sum_{i,j=1}^{n+1}\frac{\partial}{\partial
x_{i}}\left(A_{i,j} \,\frac{\partial}{\partial x_{j}}\right)\end{equation*}
is defined in $\mathbb{R}^{n+1}=\{(x,t)\in\mathbb{R}^{n}\times\mathbb{R}\}, n\geq 2$
  (we recall that we use the notational conventions that $x_{n+1}=t$, and that
capital letters may be used to denote points in $\ree$, e.g., $X=(x,t)$). 
Here, $A=A(x)$ is an $(n+1)\times(n+1)$ matrix of complex-valued $L^{\infty}$ coefficients, defined on
$\mathbb{R}^{n}$ (i.e., independent of the $t$ variable) and satisfying the 
uniform ellipticity condition
\begin{equation}
\label{eq7.1} \lambda|\xi|^{2}\leq\Re e\,\langle A(x)\xi,\xi\rangle
:= \Re e\sum_{i,j=1}^{n+1}A_{ij}(x)\xi_{j}\bar{\xi_{i}}, \quad
  \Vert A\Vert_{L^{\infty}(\mathbb{R}^{n})}\leq\Lambda,
\end{equation}
 for some $\lambda>0$, $\Lambda<\infty$, and for all $\xi\in\mathbb{C}^{n+1}$, $x\in\mathbb{R}^{n}$. 

 In the present work, we further suppose that weak solutions of the equations $Lu=0$ (and also
 $L^*u=0$, where $L^*$ denotes the Hermitian adjoint of $L$) satisfy 
De Giorgi/Nash (``DG/N")  estimates, i.e., that
 there is a constant $C$
and an exponent $\alpha>0$ such that for any ball
$B=B(X,R)\subset \ree$, of radius $R$,  for which $Lu=0$ in the concentric double $2B := 
B(X,2R)$, we have the local H\"older continuity estimate 
\begin{equation}
|u(Y)-u(Z)|\leq
C\left(\frac{|Y-Z|}{R}\right)^\alpha\left(\fint_{2B}|u|^{2}\right)^{1/2},\label{eq7.2}\end{equation}
whenever $Y,Z\in B$. Observe that any $u$ satisfying \eqref{eq7.2} also 
satisfies Moser's ``local boundedness" estimate \cite{M}
\begin{equation}
\sup_{Y\in B}|u(Y)|\leq C \left(\fint_{2B}|u|^{2}\right)^{1/2}.\label{eq7.3}\end{equation}
Estimates \eqref{eq7.2} and \eqref{eq7.3} always hold for real coefficients
\cite{DeG,N}, and are stable under small  complex, $L^{\infty}$ perturbations
\cite{A2} (see also~\cite{AT,HK2});  moreover they hold always for 
$t$-independent complex operators
in ambient dimension $n+1 = 3$  \cite[Section 11]{AAAHK}.
In the presence of the DG/N estimates (for $L$ and $L^*$), by \cite{HK},
the operators $L$ and $L^*$, respectively, have fundamental solutions
$E(X,Y)$, and $E^*(X,Y) = \overline {E(Y,X)}$, satisfying the following properties:
first, that
\begin{equation}\label{eq7.4a}
L_{x,t} \,E (x,t,y,s) = \delta_{(y,s)},\,\,\,
L^*_{y,s}\, E^*(y,s,x,t):= L^*_{y,s} \,\overline {E (x,t,y,s)} = \delta_{(x,t)},
\end{equation}
where $\delta_X$ denotes the Dirac mass at the point $X$;  
second,  by the $t$-independence of our 
coefficients,  that 
\begin{equation}
E(x,t,y,s)=E(x,t-s,y,0);\label{eq7.4}\end{equation}
and finally, 
for $j\geq 0$, that 
there exists a constant $C_j$ depending only on $j$,
dimension, ellipticity and \eqref{eq7.2} and \eqref{eq7.3}, such that for 
for all $(x,t)\neq(y,s)\in \mathbb{R}^{n+1}$, we have
\begin{equation}\label{eq7.6} \left|(\partial_t)^{j} E(x,t,y,s) \right| \leq\,
C_j \,\big(|t-s|+|x-y|\big)^{1-n-j}\end{equation}
\begin{equation}\label{eq7.7} \left|\left(\Delta_h(\partial_t)^{j} E(\cdot,t,y,0)\right)(x)\right| + 
\left|\left(\Delta_h(\partial_t)^{j} E(x,t,\cdot,0)\right)(y)\right|
\leq C_j \frac{|h|^\alpha}{(|t|+|x-y|)^{n+j+\alpha-1}},\end{equation} whenever 
$2|h|\leq |x-y|$ or $|h|<20|t|$, for some $\alpha >0$, where $\left(\Delta_h f\right)(x) 
:= f(x+h) - f(x).$

We define the single layer potential operator, associated to $L$, by
\begin{equation}\label{eq7.8}\SL_{t}f(x) :=\int_{\mathbb{R}^{n}}E(x,t,y,0)\,f(y)\,dy, \,\,\, t\in 
\mathbb{R}\,;
\end{equation}
the single layer potential associated to $L^*$ is defined analogously, with
$E^*$ in place of $E$.    

We shall encounter 
operators whose kernels involve derivatives applied to the second set of variables in the 
fundamental solution $E(x,t,y,s)$, and we denote this by appropriate parenthetic grouping;
e.g., for $f:\rn\to\CC^{n+1}$, we set
\begin{equation}\label{eq7.8a}
\big(\SL_t\nabla\big) f(x) := \int_{\rn} \left(\nabla_{y,s} E(x,t,y,s)\right)\!\!\big|_{s=0}\, \cdot
f(y)\,dy\,.
\end{equation}
We similarly denote the individual components of the vector-valued operator $(\SL_t\nabla)$, thus
for $f:\rn\to\CC^{n}$, 
\begin{equation}\label{eq7.8b}
\big(\SL_t\nabla y\big) f\,= \,-  \SL_t \left(\dv_y f\right)\,,
\end{equation}
and, by translation invariance in $t$ \eqref{eq7.4}, 
\begin{equation}\label{eq7.8c}
\big(\SL_t\partial_s\big) = - \partial_t \SL_t\,.
\end{equation}
With these notational conventions in hand, 
we record for future reference the following estimates taken from
\cite{AAAHK}:
\begin{lemma}[\cite{AAAHK}, Lemma 2.8, Lemma 2.9 and Lemma 2.10]\label{l27.1}
Suppose that $L,L^*$  are
$t$-independent divergence form complex elliptic operators as above, whose null
solutions satisfy the
DGN bounds \eqref{eq7.2}. 
Then for some $C$ depending only on $n$, \eqref{eq7.1} and \eqref{eq7.2}, 
for every fixed $x\in\rn$ and $t\neq 0$,  we have
\smallskip
\begin{equation}\label{eq7.12a}
\int_{\rn} |\nabla_{y,s} E(y,s,x,t)|_{s=0}|^2\, dy \leq C |t|^{-n}
\end{equation}
In addition, 
if  $j\geq 0$,
then there is a constant $C_j$ depending on $n$, \eqref{eq7.1} and \eqref{eq7.2}, such that
for $ f: \mathbb{R}^n \to \mathbb{C}^{n+1}$, and
for every cube $Q$,  for all integers $k\geq 1$,  and  
 for all $t\in\mathbb{R}$, 
\begin{equation}\label{eq7.11}
\|(\partial_t)^j (\SL_t\nabla )\,(f
1_{2^{k+1}Q\setminus 2 ^kQ})\|^2_{L^2(Q)} 
\leq \,C_j\, 2^{-nk} (2 ^k\ell (Q))^{-2j} \|f\|^2_{L^2(2^{k+1}Q\backslash 2^k
Q)}\,,
\end{equation}
Moreover, for each $j\geq 1$, 
\begin{equation}\label{eq7.12}
\| t^{j}(\partial_t)^j \left(\SL_t\nabla \right)f\|_{L^2(\mathbb{R}^n)}\leq 
C_j\,\|f\|_2
\end{equation}
\end{lemma}
We note that by \eqref{eq7.8c}, for scalar valued $f$
we obtain from \eqref{eq7.12} that, for $j\geq 1$,
\begin{equation}\label{eq7.13}
\| t^{j}(\partial_t)^{j+1} \SL_t f\|_{L^2(\mathbb{R}^n)}\leq \,
C_j\,\|f\|_2\,.
\end{equation}
We further note for future reference that \eqref{eq7.11}
can be reformulated as
\begin{equation}\label{eq7.14}
\Vert t^{j}(\partial_{t})^{j}(\SL_{t}\nabla)\,(f1_{2^{k+1}Q\backslash 2^{k}Q})\Vert_{L^{2}(Q)}^{2}\leq\,
C_{j}\,2^{-nk}\left(\frac{t}{2^{k}\ell(Q)}\right)^{2j}\Vert
f\Vert^2_{L^{2}(2^{k+1}Q\backslash 2^{k}Q)}\,,
\end{equation}
so in particular, 
 \begin{equation}\label{eq7.15}
\Vert t^{2}(\partial_{t})^{2}(\SL_{t}\nabla)\,(f1_{2^{k+1}Q\backslash 2^{k}Q})\Vert_{L^{2}(Q)}^{2}\lesssim
\,2^{-(n+4)k}\,\Vert
f\Vert^2_{L^{2}(2^{k+1}Q\backslash 2^{k}Q)}\,,\quad {\rm if}\,\, t\approx \ell(Q)\,,
\end{equation}
which by \eqref{eq7.8c} yields also that
 \begin{equation}\label{eq7.15a}
\Vert t^{2}(\partial_{t})^{3}\SL_{t}\,(f1_{2^{k+1}Q\backslash 2^{k}Q})\Vert_{L^{2}(Q)}^{2}\lesssim
\,2^{-(n+4)k}\,\Vert
f\Vert^2_{L^{2}(2^{k+1}Q\backslash 2^{k}Q)}\,,\quad {\rm if}\,\, t\approx \ell(Q)\,.
\end{equation}

The main result of this section is the following, which is originally due to Rosen \cite{R}.
\begin{theorem}\label{t7.1}  
Let $L$ be a $t$-independent complex elliptic operator as above,
for which solutions of $Lu=0$ and $L^*u=0$ satisfy the DG/N bounds \eqref{eq7.2}
(hence also \eqref{eq7.3}).  Then the single layer potential satisfies the following
square function bound:
\begin{equation}\label{eq7.9}
\iint_{\uh}\left|t\,(\partial_t)^2 \SL_t f(x)\right|^2 \frac{ dxdt}{t} \leq C\,\|f\|^2_{L^2(\rn)}\,.
\end{equation}
with $C$ depending only upon $n, \lambda,\Lambda$, and the constants in 
\eqref{eq7.2} and \eqref{eq7.3}.    Analogous bounds hold in the lower half space
$\mathbb{R}^{n+1}_-$, and for the single  layer potential associated to $L^*$.
\end{theorem}
This is the  fundamental result concerning the layer potentials.
Indeed, by 
 \cite[Lemma 5.2]{AAAHK}, we have
\begin{equation*}
\sup_{t\neq0}\|\nabla\SL_t f\|^2_{L^2(\rn)}\,
 \lesssim\, \|N_*\left(\partial_t\SL_t f\right)\|^2_{L^2(\rn)} \,+\,
 \int_{-\infty}^{\infty}\int_{\mathbb{R}^{n}}\left|t\partial_{t}^{2}\SL_{t}f(x)\right|^{2}\frac{dx dt}{|t|}\,+\,
 \|f\|^2_{L^2(\rn)}\,,
 \end{equation*}  
 and also
 $$\int_{-\infty}^{\infty}\int_{\mathbb{R}^{n}}\left|t\nabla\partial_t\SL_{t}f(x)\right|^{2}\frac{dx dt}{|t|}
 \lesssim \int_{-\infty}^{\infty}\int_{\mathbb{R}^{n}}\left|t\partial_{t}^{2}\SL_{t}f(x)\right|^{2}\frac{dx dt}{|t|}\,+\,
 \|f\|^2_{L^2(\rn)}\,. $$ 
 Moreover, by \cite{AA}, we have
 $$\|N_*\left(\partial_t\SL_t f\right)\|^2_{L^2(\rn)}\,\lesssim\,
 \int_{-\infty}^{\infty}\int_{\mathbb{R}^{n}}\left|t\nabla\partial_t\SL_{t}f(x)\right|^{2}\frac{dx dt}{|t|}\,.$$
 Combining these last three estimates yields,  in particular,
 the bound  
 $$\sup_{t\neq0}\|\nabla\SL_tf\|_{L^2(\rn)}\lesssim  \|f\|_{L^2(\rn)}\,.$$
Bounds for the associated double layer potential (see \cite[(1.5)]{AAAHK} for a precise definition)
 follow by duality, while boundedness of an appropriate non-tangential maximal function of 
 $\nabla\SL_t f$
follows by
\cite[Lemma 4.8]{AAAHK}.   The analogous $L^p$ and $H^p$ bounds will appear in 
\cite{HMiMo}.

Let us  proceed now to give the proof of Theorem \ref{t7.1}.  
\begin{proof}
We first make the following claim:
\begin{equation}\label{7.17}
\iint_{\uh}\left|t\,(\partial_t)^2 \SL_t f(x)\right|^2 \frac{ dxdt}{t} \lesssim 
\iint_{\uh}\left|t^2\,(\partial_t)^3 \SL_t f(x)\right|^2 \frac{ dxdt}{t} \,+\,\|f\|_{L^2(\rn)}^2
\end{equation}
Indeed, the left hand side of \eqref{7.17} equals the limit, as $\delta \to 0$, of
$$\int_\delta^{1/\delta}\!\!\int_{\rn}\left|(\partial_t)^2 \SL_t f(x)\right|^2 \, t\, dx dt
\lesssim \,\int_\delta^{1/\delta}\!\!\int_{\rn}\left|(\partial_t)^2 \SL_t f(x)\right|\,
\left|(\partial_t)^3 \SL_t f(x)\right| \, t^2\, dx dt\,+\, \|f\|_2^2\,,$$
where we have integrated by parts in $t$, and then used \eqref{eq7.13}, with $j=1$,
to control the boundary terms.  Applying ``Cauchy's inequality with $\eps$'s", hiding a small term on the left hand side of the last inequality, and then letting $\delta\to 0$, we obtain \eqref{7.17}.
Thus, in lieu of \eqref{eq7.9}, it is now enough to prove
\begin{equation}\label{eq7.18}
\iint_{\uh}\left|t^2\,(\partial_t)^3 \SL_t f(x)\right|^2 \frac{ dxdt}{t} \,\leq \,C\,\|f\|_{L^2(\rn)}^2\,.
\end{equation}

We therefore turn to the proof of \eqref{eq7.18}.
We shall utilize the following 
notation:  given a vector $\vec{V}:= (V_1,...,V_n,V_{n+1})\subset \ree$, we denote its ``horizontal
component" by
$$V_\|:= (V_1,...,V_n)\,.$$
Similarly, the horizontal component of the $(n+1)$-dimensional gradient operator is
$$\nabla_\|:= (\partial_{x_1},...,\partial_{x_n})\,.$$
For convenience, we write
$$\left(\nabla_{y,s} u\right)(y,0) := \big(\nabla_{y,s} u(y,s)\big)\!\big|_{s=0}\,.$$
We now define 
$$\Theta_t^0 := \,t^2 \,(\partial_t)^3 \SL_t\,,$$
and for $f:\rn\to \CC^n$, we set
$$\Theta'_t f(x) :=\, t^2\,(\partial_t)^2 \int_{\rn} \Big(\overline{A^*}(y)\left(\nabla_{y,s}E\right)(x,t,y,0)\Big)_{\!\|}
\cdot f(y)\, dy\,.$$
It is then enough to show that the hypotheses of  Theorem \ref{t6.1} are verified for this choice of
$\Theta_t=(\Theta_t^0,\Theta'_t)$, with $m=n$, where the ``0" term corresponds to the $t=x_{n+1}$
direction.  
We first observe that \eqref{6.5a}, and \eqref{6.6a} 
(with $\beta=2$),  follow directly from
\eqref{eq7.12}-\eqref{eq7.13} (with $j=2$), and \eqref{eq7.15}-\eqref{eq7.15a}, respectively.  We shall 
establish the quasi-orthogonality estimate \eqref{6.8}, with $m=n$,
and with $H:=\{\nabla f: f\in \dot{L}^2_1(\rn)\}$.  
We treat separately the contributions of $\Theta_t^0 Q_s h^0$, 
and of $\Theta'_tQ_sh'$,  where $h^0\in L^2(\rn,\CC)$, and $h'\in H$.  For the former,
 we use a standard ``Calder\'on-Zygmund" argument exploiting \eqref{eq7.7} (with $j=3$), 
 and the cancellation of the
convolution kernel of $Q_s$;  we leave the routine details to the reader.
Consider now $\Theta'_tQ_sh'$, where $h'=\nabla f$, with $f\in \dot{L}^2_1(\rn)$.
Recall that, modulo constants, we may identify
$f\in \dot{L}^2_1(\rn)$, with $I_1 g$, where $g\in L^2(\rn)$, and $\|g\|_2\approx \|\nabla f\|_2$,
and $I_1=(-\Delta)^{-1/2}$ is the standard fractional integral operator of order 1.
Moreover, for $(x,t)\in\uh$ fixed, and for $y_{n+1}:=s<t$, using that 
$E(x,t,y,s) =\overline{E^*(y,s,x,t)}$, we have (in the weak sense)
\begin{multline}\label{eq7.19}
-\dv_y \Big(\overline{A^*}(y)\nabla_{y,s}E(x,t,y,s)\Big)_{\!\|} =
-\sum_{i=1}^n\sum_{j=1}^{n+1} \overline{\partial_{y_i}  \,A^*_{i,j}(y)\, \partial_{y_j} E^*(y,s,x,t)}\\[4pt]
=\, \sum_{j=1}^{n+1}\overline{ \partial_{s}  \,A^*_{n+1,j}(y)\, \partial_{y_j} E^*(y,s,x,t)}\,=\,
-\sum_{j=1}^{n+1}  \,\overline{A^*_{n+1,j}(y)\, \partial_{t} \,\partial_{y_j} E^*(y,s,x,t)}\,,
\end{multline}
where we have used that $E^*(\cdot,x,t)$ is an adjoint solution away from the pole at $(x,t)$,
and, in the last step, $t$-independence of the coefficients (cf. \eqref{eq7.4}.)  Thus, for $h'=\nabla f
\in H$, again using that $\overline{E^*(y,s,x,t)}=E(x,t,y,s)$, we have, by \eqref{eq7.19} and the 
definition of $\Theta'_t$,
$$\Theta'_t Q_sh' =\Theta'_t \,\nabla_\| Q_s f 
=-t^2 (\partial_t)^3 \big(\SL_t\nabla\big) \big(\vec{\alpha} Q_s I_1 g\big)\,,$$
where $\vec{\alpha} := (A_{1,n+1},...,A_{n+1,n+1})$.  By \eqref{eq7.12} (with $j=3$), since
$\|s^{-1} Q_s I_1\|_{L^2\to L^2} \leq C$,  we then obtain \eqref{6.8}, with $\beta=1$.
We have therefore established that $\Theta_t^0$ and $\Theta'_t$ satisfy all the conditions
of Definition \ref{def6.1}.  

It remains to construct a system $\{b_Q\}\subset L^2(\rn,\CC)\times H$, and a family of
Borel measures $\{\mu_Q\}$ as in Definition \ref{def6.12}, satisfying the hypotheses 
\eqref{6.9}-\eqref{6.12} of Theorem \ref{t6.1},
and also to verify hypothesis \eqref{6.13}.  In fact, the latter estimate is known:  it has been
proved in \cite[Section 3]{HMaMo}, with constant $C_1$ depending only on dimension, ellipticity,
and the DGN constants.  We therefore turn to the heart of the matter, namely, the construction of the system $\{b_Q\}$, and the family $\{\mu_Q\}$, verifying \eqref{6.9}-\eqref{6.12}.

Given a cube $Q\subset \rn$, we let $x_Q$ denote its center, and we let
$$X_Q^\pm:= \big(x_Q,\,\pm\, \tau \,\ell(Q)\big)\,,$$
denote the upper and lower ``Corkscrew points" relative to $Q$, where
$\tau\in (0,1/8)$ is a small number at our disposal, to be chosen.  
Set
$$F_Q(y,s):= E(y,s,X_Q^+)-E(y,s,X_Q^-)\,,$$
and define
$$b_Q^0(y):= \,|Q|\,\left(\partial_{\nu^-_{A}} F_Q\right)(y,0) := \,
|Q|\,e_{n+1}\cdot \ A(y) \left(\nabla_{y,s} F_Q\right)(y,0)\,,$$
where $e_{n+1}:= (0,...,0,1)$ is the inner unit normal to the half space $\uh$
(thus,  $\partial_{\nu^-_A}$ is the outer co-normal derivative, relative to $A$, on the boundary of 
the {\it lower}
half-space $\mathbb{R}^{n+1}_-$), and
$$b'_Q(y):=\,|Q|\,\nabla_\| F_Q(y,0)\,.$$  We then have, by \eqref{eq7.12a}, with $t=\,\pm \,\tau\,\ell(Q)$,
\begin{equation}\label{eq7.22a}
\int_{\rn} |b_Q|^2 \,\leq \, C \tau^{-n} |Q|\,,
\end{equation}
which is \eqref{6.9} with $C_0\approx \tau^{-n}$.

Next, by \eqref{eq7.19}
and the definition of $\Theta_t'$,  we have that
\begin{multline}\label{eq7.22}
\Theta'_t b'_Q (x) \,=\,t^2 |Q| \,(\partial_t)^3 \int_{\rn} (-e_{n+1})\cdot 
\overline{A^*(y)\left(\nabla_{y,s}E^*\right)
(y,0,x,t)}\,F_Q(y,0)\, dy\\[4pt]
=:\, -\,t^2 |Q| \,(\partial_t)^3 \int_{\rn} \overline{\left(\partial_{\nu^-_{A^*}} E^*\right)(y,0,x,t)}
\, F_Q(y,0)\, dy\,,
\end{multline}
so that $\partial_{\nu^-_{A^*}}$ 
is the outer co-normal derivative, relative to $A^*$, for the lower half-space.
Let  $\langle\cdot,\cdot\rangle_{\mathbb{R}^{n+1}_-}$ denote the distributional pairing 
between continuous functions and measures in the lower half-space 
$\mathbb{R}^{n+1}_-$.
Combining \eqref{eq7.22} with the definitions of $b_Q^0$ and $\Theta_t^0$, 
and recalling that  
$\overline{E^*(y,s,x,t)}=E(x,t,y,s)$,  we have for each fixed $(x,t)\in \uh$,
\begin{multline*}
\Theta_t b_Q(x) =\, \Theta_t^0b^0_Q(x)\,+\,\Theta'_tb'_Q(x)\\[4pt]=\, 
t^2 |Q| \,(\partial_t)^3 \int_{\rn} \left(\overline{E^*(y,0,x,t)}\,\left(\partial_{\nu^-_{A}} F_Q\right)(y,0)\,-\,
\overline{\left(\partial_{\nu^-_{A^*}} E^*\right)(y,0,x,t)}
\, F_Q(y,0)\right) \, dy\\[4pt]
=\,-\, t^2 |Q| \,(\partial_t)^3 
\Big(\left\langle \overline{E^*(\cdot,\cdot,x,t)}, LF_Q\right\rangle_{\mathbb{R}^{n+1}_-}
\,-\,\iint_{\mathbb{R}^{n+1}_-}\overline{L^* E^*(y,s,x,t)}\, F_Q(y,s) \, dyds\Big)\\[4pt]
= \, t^2 |Q| \,(\partial_t)^3 \overline{E^*(X_Q^-,x,t)} =\, t^2 |Q| \,(\partial_t)^3 E(x,t,X_Q^-) \,,
\end{multline*}
where we have used the definition of $F_Q$, \eqref{eq7.4a}, and the fact that, in the lower half-space, 
$L^*E^*(\cdot,\cdot, x,t)=0=LE(\cdot,\cdot, X_Q^+)$.  The preceding formal argument may be justified by introducing a smooth cut-off adapted to a ball of radius $R$, and eventually letting $R\to\infty$;
there is sufficient decay at infinity to justify the limiting procedure.    
Therefore, by
\eqref{eq7.6}, with $j=3$, we have that for $(x,t) \in R_Q:=Q\times (0,\ell(Q))$,
$$|\Theta_t b_Q(x)| \,\lesssim \, \tau^{-n-2} \left(\frac{t}{\ell(Q)}\right)^{2}\,.$$
Consequently,
$$\iint_{R_Q} |\Theta_t b_Q(x)|^2 \frac{dx dt}{t}
 \,\lesssim \, \tau^{-2n-4}\, |Q| \int_0^{\ell(Q)}\left(\frac{t}{\ell(Q)}\right)^{4}\frac{dt}{t}\,\approx\,
 \tau^{-2n-4}\, |Q|\,,$$
 which yields \eqref{6.10} with $C_0 \approx \tau^{-2n-4}$.  
 
 We now turn to the proofs of \eqref{6.11}-\eqref{6.12}.  We begin by defining the measure
 $\mu_Q$.  Let $\omega$ be a small, positive constant, to be chosen, and let $\phi_Q:\rn\to[0,1]$
 be a smooth bump function, supported in $(1+\omega)Q$, with $\phi_Q\equiv 1$ on $(1/2)Q$,
 where for any positive constant $\kappa$, we let $\kappa Q$ denote the cube 
 of side length $\kappa\ell(Q)$, concentric with $Q$.  Clearly, we may choose $\phi_Q$
 so that $\|\nabla\phi_Q\|_{L^\infty(\rn)}\leq 2\ell(Q)^{-1}$, with $\phi_Q\gtrsim \omega$
 on $Q$;  thus \eqref{6.13a} holds with $C_0 = 2$, and \eqref{6.14a} 
 holds with $C_0 \approx 1/\omega$.   In accordance with Definition 
\ref{def6.12}  we then set $d\mu_Q:= \phi_Q dx$.
Let $\Phi_Q:\mathbb{R}^{n+1}\to [0,1]$ be a smooth extension of
of $\phi_Q$, i.e., $\Phi_Q(x,0)= \phi_Q(x)$, with $\Phi_Q$ supported
in $I_{(1+\omega)Q}$, and
$\Phi_Q\equiv 1$ in $I_{(1/2)Q}$, where in general, for any cube $Q\subset \rn$, we let
$I_Q:= Q\times (-\ell(Q),\ell(Q))$ denote the ``two-sided Carleson box".
We choose the extension $\Phi_Q$ so that
\begin{equation*}
\|\nabla\Phi_Q\|_{L^\infty(\mathbb{R}^{n+1})}\leq\,2\,\ell(Q)^{-1}\,.
\end{equation*}

We note that, by \eqref{eq7.22a},
\begin{equation}\label{eq7.24}
\int_{\rn\setminus Q}|b_Q|\,\phi_Q
\,\leq\, \big((1+\omega)Q\setminus Q\big)^{1/2} \|b_Q\|^2_{L^2(\rn)}
\,\leq\,C \, \omega^{1/2} \tau^{-n/2}\, |Q|\,.
\end{equation}
Observe also that 
\begin{multline*}
\int_{\rn} b_Q^0\, d\mu_Q \,=\,\int_{\rn} b_Q^0\, \phi_Q\, =\,
|Q| \int_{\rn} \left(\partial_{\nu^-_A}F_Q\right)(y,0)\,\phi_Q(y)\, dy\\[4pt]
=\, |Q|\left(-\,\left\langle\Phi_Q, LF_Q\right\rangle_{\mathbb{R}^{n+1}_-}\,+\,\iint_{\mathbb{R}^{n+1}_-}A\nabla F_Q\cdot\nabla \Phi_Q\right)\,=:\,|Q|\,\big(I + II\big)\,.
\end{multline*}
By definition of $F_Q$ and $\Phi_Q$, 
$$I \,= \,\Phi_Q(X_Q^-)\,=\, 1\,.$$
Also, by the construction of $\Phi_Q$, we have that
\begin{multline*}
|II| \lesssim \frac1{\ell(Q)} \iint_{I_{(1+\omega)Q}\setminus I_{(1/2)Q}} |\nabla F_Q|\\[4pt]
\lesssim\, \ell(Q)^{(n-1)/2}\left( \iint_{I_{(1+\omega)Q}\setminus I_{(1/2)Q}} |\nabla F_Q|^2\right)^{1/2}
\lesssim\,  \ell(Q)^{(n-3)/2}\left( \iint_{I_{2Q}\setminus I_{(1/4)Q}} | F_Q|^2\right)^{1/2}\\[4pt]
=\,  \ell(Q)^{(n-3)/2}\left( \iint_{I_{2Q}\setminus I_{(1/4)Q}} \left|\int_{-\tau\ell(Q)}^{\tau\ell(Q)}
\partial_t E(y,s,x_Q,t)\, dt \right|^2dyds\right)^{1/2}\,\lesssim\, \tau\,,
\end{multline*} 
where the implicit constants depend only on dimension, ellipticity, and the DGN constants, and
where in the last three steps we have used Caccioppoli's inequality, the definition of $F_Q$, and
\eqref{eq7.6} with $j=1$.  Combining our estimates for terms $I$ and $II$, we obtain
$$\mathcal{R}e \int_{\rn} b_Q^0\, d\mu_Q \,\geq\,|Q|\big(1-C\tau\big)\,.$$  
In conjunction with \eqref{eq7.24}, the latter bound implies
$$\mathcal{R}e \int_{Q} b_Q^0\, d\mu_Q\, \geq\,|Q|\big(1-C\tau -C\omega^{1/2} \tau^{-n/2}\big)\,
\geq\, \frac12 |Q|\,,$$  
if we set $\omega:= \tau^{n+2}$, and choose $\tau$ sufficiently small.
Thus, we obtain \eqref{6.11} with $\sigma=1/2$.  

Finally, we verify \eqref{6.12}.  By  definition of $b_Q'$, $\mu_Q$, $\phi_Q$, and $F_Q$,
we have
\begin{multline*}
\left|\int_{\rn} b_Q' d\mu_Q\right| \, = |Q|\,\left|\int_{\rn} \nabla_y F_Q(y,0)\, \phi_Q(y)\, dy \right|
\,=\,  |Q|\,\left|\int_{\rn}  F_Q(y,0) \,\nabla_y \phi_Q(y)\, dy \right|\\[4pt]
\leq C\,\ell(Q)^{n-1}\,\int_{(1+\omega)Q\setminus (1/2)Q}  \left|\int_{-\tau\ell(Q)}^{\tau\ell(Q)}
\partial_tE(y,0,x_Q,t)\, dt\right|\, dy \,\leq\, C\tau\, |Q|\,,
\end{multline*} 
where in the last step we have used \eqref{eq7.6} with $j=1$.  Combining the latter 
estimate with  \eqref{eq7.24}, we have
$$\left|\int_{Q} b_Q' d\mu_Q\right| \,\leq C\big(\tau +\omega^{1/2} \tau^{-n/2} \big)\,|Q|\,
\leq \, C\tau\,|Q|\,,$$
by our choice of $\omega=\tau^{n+2}$.  Since the constant $\tau$ is at our disposal
(cf. Remark \ref{remark6.3}), 
we obtain \eqref{6.12}, with $\eta\approx \tau$.  This concludes the proof of our application,
Theorem \ref{t7.1}.
\end{proof}

\section{Appendix:  
proof of the generalized Christ-Journ\'e $T1$ Theorem for square functions}\label{st1}

In this Appendix, we give the proofs of Theorems \ref{3.T1theorem} and \ref{4.T1theorem}.
In fact, it will suffice to prove the latter, since taking $H=L^2$ then yields the former.
The proof will follow that of the $T1$ theorem of \cite{CJ};  indeed, the lack of pointwise
kernel bounds, and the fact that our square function acts only on
elements of $H$, do not present serious obstacles.

\begin{proof}[Proof of Theorem \ref{4.T1theorem}]
We consider
\begin{equation*}
\iint_{\uh}|\Theta_t h(x)|^2\frac{dxdt}{t}\,,
\end{equation*}
with $h\in H$, where $H$ is a subspace of $L^2(\rn,\CC^m)$.   
Let us set
$$ d\mu(x,t):= |\Theta_t 1(x)|^2\frac{dxdt}{t}\,,$$
so that, by hypothesis, we have the Carleson measure estimate
$$\|\mu\|_{\mathcal{C}}:= \sup_Q\frac1{|Q|}\iint_{R_Q} d\mu(x,t) <\infty\,.$$
We follow the
familiar idea of Coifman-Meyer to write
$$\Theta_t h\,=\,\Big(\Theta_t -(\Theta_t1)A_t P_t\Big) h
 \,+\, (\Theta_t1) A_tP_th\,=:\, R_t h +(\Theta_t1) P_th\,,$$
where as usual, $A_t$ is the dyadic averaging operator, and 
$P_t$ is a nice approximate identity.  By Carleson's embedding lemma,
and the well known non-tangential estimate for $A_tP_t h$, 
$$\iint_{\uh}|(\Theta_t 1)A_tP_th(x)|^2\frac{dxdt}{t} \,\lesssim \,\|\mu\|_{\mathcal{C}}\|h\|_2^2\,,$$
as desired (we remark that for this term, the estimate holds for all $h \in L^2(\rn,\CC^m)$.)
Moreover, 
$$\iint_{\uh}|R_t h(x)|^2\frac{dxdt}{t} \,\lesssim \, \|h\|_2^2\,,\qquad \forall h\in H\,,$$
by Corollary \ref{c3.11}:  indeed,
the present $R_t$ is precisely the same (up to a minus sign) as $R_t^{(2)}$ considered 
in \eqref{eq3.remainderdef}.
\end{proof}


\begin{thebibliography}{99}\label{bib}



\bibitem[AAAHK]{AAAHK} M. A. Alfonseca, P. Auscher, A. Axelsson, S. Hofmann, and S. Kim,
{\it Analyticity of Layer Potentials and $L\sp 2$ Solvability of Boundary Value Problems for Divergence Form Elliptic Equations with Complex $L\sp \infty$ Coefficients}, 
 Advances in Math. 226 (2011), 4533-4606..

\bibitem[A] {A} P. Auscher, {\it On necessary and sufficient conditions for $L^p$-estimates of Riesz Transforms associated  to elliptic operators on $\mathbb{R}^n$ and related  estimates}, Memoirs of the American Mathematical Society, volume 186, number 871, March 2007.

\bibitem[A2]{A2} P. Auscher, Regularity theorems and heat kernel for elliptic operators. {\it J. London Math. Soc.} (2) {\bf 54}  (1996),  no. 2, 284--296.

\bibitem[AA]{AA} P. Auscher and A. Axelsson,  Weighted maximal regularity estimates
and solvability of non-smooth elliptic systems I, {\it  Invent. Math.} 184 (2011),  47--115.

\bibitem[AAMc]{AAMc} P. Auscher, A. Axelsson, and A. McIntosh,
Solvability of elliptic systems with square integrable boundary data,
{\it Ark. Mat.} {\bf 48} (2010), 253-287.

\bibitem[AHLMcT]{AHLMcT} P. Auscher, S. Hofmann, M. Lacey, A. McIntosh, and P. Tchamitchian, {\it The solution of the Kato Square Root Problem for Second Order Elliptic operators on $\rn$}, Annals of Math. 156 (2002), 633–654.

\bibitem [AHMTT]{AHMTT} P. Auscher, S. Hofmann, C. Muscalu, T. Tao, C. Thiele, {\it Carleson measures,trees,extrapolation, and $T (b)$ theorems}, Publ. Mat. 46 (2002), no. 2, 257–325.

\bibitem [AMcT] {AMcT} P. Auscher, A. McIntosh and P. Tchamitchian, {\it Heat kernel of complex elliptic operators and applications}, J. Funct. Anal. \textbf{152}(1998), 22-73.

\bibitem [AR] {AR} P. Auscher, E. Routin, {\it Local Tb Theorems and Hardy inequalities}, http://arxiv.org/abs/1011.1747.

\bibitem [AT] {AT} P. Auscher, P. Tchamitchian, {\it Square root problem for divergence operators and related topics}, Ast\'{e}risque Vol. \textbf{249}(1998)m soci\'{e}t\'{e} Math\'{e}matique de France.

\bibitem [AY] {AY} P. Auscher, Qi Xiang Yang, {\it BCR algorithm and th T(b) theorem}, Publ. Math.\textbf{53} (2009), no. 1, 179196.


\bibitem[Ch]{Ch} M. Christ, {\it A $T(b)$ theorem with remarks on analytic capacity and the Cauchy integral}, Colloquium Mathematicum LX/LXI (1990) 601-628.

\bibitem[CJ]{CJ} M. Christ and  J.-L. Journ\'{e}, {\it Polynomial growth estimates for multilinear singular integral operators}.   Acta Math.  {\bf 159}  (1987),  no. 1-2, 51--80.

\bibitem[CMS]{CMS} R. R. Coifman, Y. Meyer, and E. M. Stein,
{\it Some New Function Spaces and Their Applications to Harmonic Analysis}, J. Funct. Anal. 62 (1985), 304--335.

\bibitem[CR]{CR} R.Coifman and R. Rochberg, {\it Another characterization of B.M.O.}, Proc. Amer. Math. Soc \textbf{79},249-254.

\bibitem[C-UMP]{C-UMP} D. Cruz-Uribe, J.M Martell and C. P\'{e}rez, {\it Extensions of Rubio de Francia's extrapolation theorem},Proceedings of the 7th International Conference on Harmonic Analysis and Partial Differential Equations (El Escorial 2004), Collect. Math.2006,195-231..

\bibitem[DJS]{DJS}	G. David, J.-L. Journ\'{e}, and S. Semmes, {\it Op\'{e}rateurs de Calder\'{o}n-Zygmund, fonctions para-accr\'{e}tives et interpolation}, Rev. Mat. Iberoamericana 1 1–56, 1985.

\bibitem[D] {D} E.B. Davies, {\it Uniformly elliptic operators with measurable coefficients}, J. Funct. Anal. \textbf{132} (1995), 141-169.

\bibitem[DRdeF]{DRdeF} J. Duoandikoetxea, J. L. Rubio de Francia, 
{\it Maximal and singular integral operators via Fourier transform estimates,} 
Inventiones Math 84 (1986), 541-561.

\bibitem[DeG]{DeG} E. De Giorgi, {\it Sulla differenziabilit\`a e l'analiticit\`a delle estremali degli integrali multipli regolari}, Mem. Accad. Sci. Torino. Cl. Sci. Fis. Mat. Nat. (3) 3  (1957) 25--43.


\bibitem[FKP]{FKP}  R. Fefferman, C. Kenig and J. Pipher,
{\it The theory of weights and the Dirichlet problem for elliptic equations}, Ann. of Math. (2)  134 (1991),  no. 1, 65--124. 

\bibitem[FS]{FS} C.  Fefferman and E. M. Stein,  $H\sp{p}$ spaces of several variables.  
{\it Acta Math.}  {\bf 129}  (1972), no. 3-4, 137--193.



\bibitem[GR]{GR} J. Garcia Cuerva and J. Rubio de Francia, {\it Weighted norm inequalities and related topics}, North-Holland, 1985.

\bibitem[GM]{GM} A. Grau de la Herr\'{a}n and M. Mourgoglou, {\it A Tb theorem for square functions in domains with Ahlfors-David regular boundaries}, to appear in the Journal of Geometric Analysis.


\bibitem[H1]{H1} S. Hofmann, {\it Local T(b) theorems and application in PDE}, Harmonic analysis and partial differential equations 29-52, Contemp. Math, 505, Amer. Math. Soc., Providence, RI, 2010.

\bibitem[H2]{H2} S. Hofmann, {\it A Local $Tb$ Theorem For Square Functions}, Perspectives in partial differential equations, harmonic analysis and applications, 175–185, Proc. Sympos. Pure Math., 79, Amer. Math. Soc., Providence, RI, 2008.

\bibitem[H3]{H3} S. Hofmann, {\it A proof of the local Tb theorem for standard Calderon-Zygmund operators}, unpublished manuscript (2007), http://arxiv.org/abs/0705.0840.2

\bibitem[H4]{H4} S. Hofmann, {\it Local $Tb$ Theorem For Square Functions and Application in PDE}, Proceedings of the ICM Madrid 2006.

\bibitem[HKMP]{HKMP} S. Hofmann, C. Kenig, S. Mayboroda, and J. Pipher, {\it Square function/Non-tangential maximal
estimates and the Dirichlet problem for
non-symmetric elliptic operators,}   preprint. 

\bibitem[HK]{HK} S. Hofmann and S. Kim, {\it The Green function estimates for strongly elliptic systems of second order}, Manuscripta Math. 124 (2007), no. 2, 139-172.

\bibitem[HK2]{HK2} S. Hofmann and S. Kim, Gaussian estimates for fundamental solutions to certain parabolic systems, {\it Publ. Mat.} {\bf 48}  (2004),  no. 2, 481--496.


\bibitem[HLMc]{HLMc} S. Hofmann, M. Lacey and A. McIntosh, {\it The solution of the Kato problem for divergence form elliptic operators with Gaussian heat kernel bounds}, Annals of Math. 156 (2002), pp 623-631.

\bibitem[HMaMo]{HMaMo} S. Hofmann, S. Mayboroda, and M. Mourgoglou,
$L^p$ and endpoint solvability results for divergence form elliptic
equations with complex $L^{\infty}$ coefficients, preprint.

\bibitem[HMar]{HMar} S.Hofmann and J.M. Martell, {\it Uniform rectifiability and harmonic measure I:
Uniform rectifiability implies Poisson kernels in $L^p$}, preprint.

\bibitem[HMarUT]{HMarUT} S.Hofmann, J.M. Martell and I. Uriarte-Tuero, {\it Uniform rectifiability and harmonic measure II:  Poisson kernels in $L^p$ imply uniform rectfiability}, preprint.

\bibitem[HMc]{HMc} S. Hofmann and A. McIntosh, {\it The solution of the Kato problem in two dimensions}, Proceedings of the Conference on Harmonic Analysis and PDE held in El Escorial, Spain in July 2000, Publ. Mat. Vol. extra, 2002 pp. 143-160.

\bibitem[HMiMo]{HMiMo} S. Hofmann, M. Mitrea, and A. Morris, manuscript in preparation.

\bibitem[HM]{HM} S. Hofmann and M. Mitrea, {\it Boundary Value Problems and the Method of Layer Potentials for Elliptic Operators with $L^{\infty}$ Coefficients} preprint.

\bibitem[HyM]{HyM} T. Hyt\"onen, H. Martikainen, {\it On general local Tb Theorems}, http://arxiv.org/abs/1011.0642

\bibitem[HyN]{HyN} T. Hyt\"onen, F. Nazarov, {\it
The local Tb theorem with rough test functions}, http://arxiv.org/abs/1206.0907 .







\bibitem[McM]{McM} A. McIntosh and Y. Meyer, {\it Alg\`ebres d’op\'{e}rateurs d\'{e}finis par des int\'{e}grales singuli\`eres}, C. R. Acad. Sci. Paris 301 S\'{e}rie 1 395–397, 1985.

\bibitem[Me]{Me}  N. G. Meyers, {\it Mean oscillation over cubes and H\"{o}lder continuity}, Proc. Amer. Math. Soc. 15 (1964) 717--721.

\bibitem[Me2]{Me2}  N. G. Meyers, {\it An $L^p$ estimate for the gradient
of solutions of second order
elliptic divergence equations}, 
Ann. Scuola Norm. Sup. Pisa 17 (1963), 189-206.



\bibitem[M]{M} J. Moser, On Harnack's theorem for elliptic differential equations, {\it Comm. Pure Appl. Math.} {\bf 14}  (1961) 577--591.

\bibitem[N]{N} J. Nash, Continuity of solutions of parabolic and elliptic equations, {\it Amer. J. Math.} {\bf 80}  (1958) 931--954.

\bibitem[NTV]{NTV} F. Nazarov, S. Treil and A. Volberg, {\it Accretive system Tb-theorems on nonhomogeneous spaces}, Duke Math. J. 113 (2002), no. 2, 259–312.

\bibitem[R]{R} A. Rosen, {\it Layer potentials beyond singular integral operators}, preprint

\bibitem[Se]{Se} S. Semmes, {\it Square function estimates and the T(b) Theorem}. Proc. Amer. Math. Soc. 110 (1990), no. 3, 721–726.

\bibitem[St]{St}  E. M. Stein, {\it Harmonic Analysis: Real-Variable Methods, Orthogonality, and Oscilatory Integrals}, Princeton, NJ: Princeton University Press, 1993.

\bibitem[St2]{St2} E.M. Stein, {\it Singular integrals and differentiability properties of functions}, Princeton, NJ: Princeton University Press, 1971.

\bibitem[SW]{SW} E.M. Stein and G. Weiss, {\it Interpolation of operators with change of measures}, Trans. Amer. Math. Soc. \textbf{87},(1958),159-172.

\bibitem[TY]{TY} C.Tan, L. Yan, {\it Local Tb theorem on spaces of homogeneous type}, Z. Anal. Anwend.\textbf{28} (2009), no.3, 333-347.

\bibitem[V]{V} G. Verchota, Layer potentials and regularity for the Dirichlet problem for Laplace's equation in Lipschitz domains. {\it J. Funct. Anal.} {\bf 59}  (1984),  no. 3, 572--611.

\end{thebibliography}
\end{document}